\documentclass[reqno]{amsart}
\usepackage{amsmath,amssymb,epsfig,graphicx}
\usepackage{hyperref}
\usepackage[hmarginratio=1:1, bottom=3cm]{geometry}

\usepackage[latin1]{inputenc} 
\usepackage{color}
\usepackage{enumerate}
\usepackage{bbm}
\usepackage{setspace}
\usepackage{subfig}
\usepackage{subfloat}
\usepackage{enumitem}
\usepackage{enumerate}
\usepackage{cancel}
\usepackage[normalem]{ulem}
\usepackage{bbm}

\newtheorem{teo}{Theorem}[section]
\newtheorem{remark}{Remark}
\newtheorem{prop}{Proposition}[section]

\newtheorem{lem}[teo]{Lemma}

\newcommand{\VF}{W}
\newcommand{\tVF}{\widetilde{W}}

\newcommand{\cA}{{\mathcal A}}

\newcommand\R{{\mathbb R}}
\renewcommand\P{{\mathbb P}}

\newcommand\E{{\mathbb E}}

\newcommand\N{{\mathbb N}}
\newcommand\Z{{\mathbb Z}}

\definecolor{newgreen}{rgb}{0,0.6,0.3}

\newcommand\sig{{\sigma}}

\newcommand{\dx}{\Delta x}

\newcommand{\beqn}{\begin{equation}}
\newcommand{\eeqn}{  \end{equation}}
\newcommand{\beqno}{\begin{equation*}}
\newcommand{\eeqno}{  \end{equation*}}
\newcommand{\be}{\begin{eqnarray}}
\newcommand{\ee}{  \end{eqnarray}}
\newcommand{\beno}{\begin{eqnarray*}}
\newcommand{\eeno}{  \end{eqnarray*}}

\newcommand{\conv}{\rightarrow}
\numberwithin{equation}{section}
\usepackage{floatrow}

\begin{document}

\author[Athena Picarelli]{Athena Picarelli}
\address{Department of Economics, University of Verona, Via Cantarane 24, 37129, Verona, Italy}
\email{athena.picarelli@univr.it}
\author[Christoph Reisinger]{Christoph Reisinger}
\address{Mathematical Institute, University of Oxford, Andrew Wiles Building, OX2 6GG, Oxford, UK}
\email{christoph.reisinger@maths.ox.ac.uk}

\title[Duality-based a posteriori error estimates]{Duality-based \textit{a posteriori} error estimates for some approximation schemes for {optimal investment} problems}
\maketitle 

\begin{abstract}
We consider a Markov chain approximation scheme for utility maximization problems in continuous time, which uses, in turn, a piecewise constant policy approximation, Euler-Maruyama time stepping, and a Gau{\ss}-Hermite approximation of the Gau{\ss}ian increments.
The error estimates previously derived in \textit{A.~Picarelli and C.~Reisinger, Probabilistic error analysis for some approximation schemes to optimal control problems, arXiv:1810.04691} are asymmetric between lower and upper  bounds due to the control approximation and improve on known results in the literature in the lower case only.
In the present paper, 
we use duality results to obtain  \textit{a posteriori} upper error bounds which are empirically of the same order as the lower bounds.
The theoretical results are confirmed by our numerical tests.
\end{abstract}


\section{Introduction}\label{sect:intro}


We study the numerical approximation of a class of optimal control problems for diffusion processes arising in financial applications.
It is well known that, under suitable assumptions, the associated value function can be characterized as the solution of a second order Hamilton-Jacobi-Bellman (HJB) partial differential equation.
To deal with the possible degeneracy of the diffusion component of the dynamics, it is in general  necessary to consider solutions in  the viscosity sense (see \cite{CIL92} for an overview). Furthermore, explicit solutions for this type of nonlinear equations are rarely available, so that their numerical approximation becomes vital.
In the framework of viscosity solutions,  the basic theory of convergence for numerical schemes is established in  \cite{BS91}.
The fundamental properties required are: monotonicity, consistency, and stability of the scheme.
While standard finite difference schemes are in general non-monotone, 
semi-Lagrangian (SL) schemes (see \cite{M89,CF95,DJ12}) 
are monotone by construction.
The basic scheme considered in this paper belongs to this family and has been previously analyzed in \cite{PicaReisi18_prob}.

{
We focus here on computable error bounds for the solution.
Many of the published error bounds for this kind of maximisation problem, including those in \cite{PicaReisi18_prob}, are asymmetrical in the sense that a more accurate lower bound can be given than the upper bound.
In this work, we construct an upper bound which consists of two additive contributions: a term which can be computed \textit{a priori} from the model parameters and is
of the same order in the mesh parameters as the known lower bounds; and a term which can be computed \textit{a posteriori} from the solution of the dual problem.
The practical value of this decomposition is that the second term is empirically (i.e., from our numerical tests) smaller than the first one, so that in practice we can compute rigorous error bounds \textit{a posteriori} which improve on the ones available \textit{a priori}. We discuss this in more detail below.

The machinery for \textit{a priori} bounds for HJB equations is now well-established.
}
By a technique pioneered by Krylov based on ``shaking the coefficients'' and mollification to construct smooth sub- and/or super-solutions,
\cite{K97, K00, BJ02, BJ05, BJ07} prove certain fractional convergence orders  significantly lower than one.  These results are mainly derived by  PDE techniques
and strongly rely on the comparison principle between viscosity sub- and super-solutions of the HJB equation and the consistency properties of the scheme. 
For the scheme considered in the present paper, 
the probabilistic proof in  \cite{PicaReisi18_prob} exploits the fact that the numerical scheme is based on a discrete approximation of the optimal control problem,
specifically by a piecewise constant policy approximation, Euler-Maruyama time stepping, and a Gau{\ss}-Hermite approximation of the Gau{\ss}ian increments.
This yields the desired error bounds by a direct comparison between two value functions and leads to an improvement of the error contribution
of the second and third of these approximations by avoiding the use of the truncation error.
The piecewise constant policy approximation, however, introduces an asymmetry between the upper and the lower bound of the error and, as a result, the bounds in \cite{PicaReisi18_prob}  give only a partial improvement of the classical PDE-based results.

For the class of convex optimal control problems studied here, namely typical utility maximization problems arising in financial applications, we propose to overcome this issue using information coming from a dual problem. Indeed, an important part of the classical literature dealing with financial applications of optimal control theory
(see the seminal work of Kramkov and Schachermayer \cite{KraScha99}) applies duality techniques to solve utility maximization problems
under suitable convexity assumptions.
The basic idea of this method is to write the optimal control problem as a constrained optimization problem with respect to the state variable and then solve it by convex analysis techniques. 
A systematic approach to utility maximization problems admitting a dual formulation is discussed in \cite{Rogers}. 
Of these, the fairly general set-up of an optimal investment problem involving nonlinear dynamics given in \cite{CuoLiu00} will be explicitly analyzed in this paper. 

More specifically, a direct application of the results in \cite{PicaReisi18_prob} to this problem gives one-sided (lower) error bounds for the considered Markov chain approximation of order 
\be\label{eq:intro_est}
h^{(M-1)/2M} + \Delta x^{(M-1)/(3M-1)}
\ee
for timestep $h$, spatial mesh size $\Delta x$ and number of Gau\ss ian points $M$, for Lipschitz viscosity solutions.
They coincide with the two-sided bounds in \cite{DJ12} for the standard linear-interpolation SL scheme, i.e.\ $M=2$, and improve them for $M>2$.
In contrast, the piecewise constant policy approximation introduces an extra term in the upper bound of order $h^{1/4}$
(from a recent result in \cite{JakoPicaReisi18}), which strictly restricts the order for $M>2$.

The main contribution of this paper is to analyse the error estimates in the case of optimal investment problems. Their special structure has neither been exploited by the classical literature on  PDE-based error estimates for HJB equations nor by the analysis in \cite{PicaReisi18_prob}. We prove that  for the class of  problems analyzed here,
two-sided \textit{a posteriori} bounds of the empirical order \eqref{eq:intro_est} can be obtained.
As a side result, we complete the literature by deriving explicit values for the constants appearing in the error estimates in terms of the Lipschitz (resp.\ H{\"o}lder) regularity of the coefficients and the solution in space (resp.\ time).

The paper is organised as follows. 
In Section \ref{sec:setting}, we introduce the problem set-up and state our assumptions.
We define the scheme and give \textit{a priori} lower error bounds for the primal problem in Section \ref{sect:disc_scheme},
and both \textit{a priori} and \textit{a posteriori} upper bounds, by way of the dual problem, in Section \ref{sec:dual_problem}.
We illustrate the theoretical results by numerical tests in Section \ref{sec:num}, and offer conclusions and extensions in
Section \ref{concl}.
In Appendix \ref{app}, we derive explicit expressions for the constants in the error  bounds.

\section{Main assumptions and preliminary results}\label{sec:setting} 

Let $\left(\Omega, \mathbb F, \P\right)$ be a probability space with filtration  $\left\{\mathbb F_t, t\geq0\right\}$ induced by a
 $d$-dimensional Brownian motion $B$ and let $T>0$.
We consider a controlled (scalar) process governed by a dynamics of the following form, {for $t\in [0,T)$},
\begin{equation}\label{eq:SDE}
\left\{\begin{array}{l}
\mathrm{d} X_s=X_s\Big(r(s) + \alpha_s^\top(b(s)-r(s)\mathbbm 1) + g(s,\alpha_s)\Big)\,\mathrm{d}s+X_s\alpha_s^\top\sig(s) \,\mathrm{d} B_s, \qquad s\in(t, T)\\
X_t=x\geq 0,\end{array}\right.
\end{equation}
where $r, b, g$ and $\sig$ take values, respectively, in $\R, \R^d, \R$ and $\R^{d\times d}$
and $\mathbbm 1\equiv (1,\ldots,1)^\top\in\R^d$.
Denote further by $\mathcal A$ the set of control policies, i.e.\ progressively measurable processes $\alpha$ taking values in a given set $A\subseteq\R^d$ such that $\int^T_0|\alpha_s|^2\mathrm{d}s<+\infty$.
{
This framework has been introduced and studied in \cite{CuoLiu00}, and encompasses a number of important optimal investment problems involving nonlinear dynamics,
including the classical Merton problem \cite{Merton}, as special cases.}
In such models, the state $X_\cdot$ typically represents the wealth of an investor {with initial endowment $x$ at time $t$}.
The control vector $\alpha\equiv(\alpha_1,\ldots,\alpha_d)^\top$ then determines the proportion of wealth the investor puts in each stock. 
Here, the coefficient $r$ is the return rate of a bond (riskless asset), while $b(\cdot)\equiv (b_1(\cdot),\ldots, b_d(\cdot))^\top$ is the vector of the appreciation rates of the $d$ considered stocks with volatility matrix $\sig(\cdot)$.
The nonlinearity in the investment strategy introduced by the function $g$ models the effects of  market frictions and trading constraints on the wealth (see \cite{CuoLiu00, CuoCvi98, ElKarPengQue97}). We refer the reader to \cite{Rogers} for an overview of different  utility maximization problems, including \eqref{eq:SDE} and its special cases.  
We consider the following assumptions:
\begin{itemize}
\item[\bf{(H1)}] $A\subseteq \R^d$ is a {bounded and} convex set such that $0\in A$.
\item[\bf{(H2)}] $(i)$
 There exists $K_0\geq 0$ such that 
\begin{eqnarray*}
& &|r(t)-r(s)|+|b(t)-b(s)|+\|\sig(t)-\sig(s)\|\leq K_0  |t-s|^{1/2}\qquad\forall\; t,s\in [0,T].
\end{eqnarray*}
$(ii)$  $g:[0,T]\times A\to \R$ satisfies: 
\begin{itemize}
\item[-] there exists $K_1\geq 0$ such that
\begin{eqnarray*}
& & |g(t,a)-g(t,a')|\leq K_1 |a-a'|\qquad \forall a,a'\in A,t\in[0,T];\\
& &|g(t,a)-g(s,a)|\leq K_1 |t-s|^{1/2}\qquad\forall\; t,s\in [0,T], a\in A;
\end{eqnarray*}
\item[-] for each $t\in[0,T]$, $a\to g(t,a)$ is concave;
\item[-] $g(t,0)=0$ for all $t\in [0,T]$.
\end{itemize}
\item[\bf{(H3)}]
 $\sig$ satisfies a uniform ellipticity condition, i.e.\ there exists $\eta>0$ such that 
$$
\xi^\top\sig\sig^\top\xi\geq \eta|\xi|^2 \qquad \forall \xi\in \R^d.
$$
\end{itemize}
One has the following existence and uniqueness result:
\begin{lem}
Let assumptions (H1) to (H3) be satisfied. For any choice of the control $\alpha\in \cA$ and $x\geq 0$ there exists a unique strong solution to  equation \eqref{eq:SDE}. 
\end{lem}
\begin{proof}
For $x>0$, {a} solution can be defined as $X_{\cdot} =\exp(Z_{\cdot})$, where
$$
Z_{\cdot} = z + \int^{\cdot}_t r(s) + \alpha_s^\top(b(s)-r(s)\mathbbm 1) + g(s,\alpha_s) -\frac{1}{2}(\alpha_s^\top \sigma)^2 \mathrm{d} s +\int^{\cdot}_t \alpha_s^\top\sig(s) \,\mathrm{d} B_s,
$$
for  $z=\log x$, which is well defined under assumptions (H1)-(H3) for any $\alpha\in\mathcal A$. Moreover, for $x=0$ the process $X\equiv 0$ is the unique solution to \eqref{eq:SDE} for any $\alpha\in\mathcal A$. 
\end{proof}
We denote  by  $X^{t,x,\alpha}_{\cdot}$ the unique solution of equation \eqref{eq:SDE}. 
To simplify the notation, where no ambiguities arise, we will indicate the starting point $(t,x)$ of the processes involved as a subscript in the expectation, i.e. $\E_{t,x}[\cdot]$.

The value function $v:[0,T]\times [0,+\infty)\to\R$ of  the optimal control problem is defined by 
\begin{align}
\label{eq:value_v}
v(t,x):=\sup_{\alpha\in\cA}\,\E_{t,x}\big[U(X^\alpha_T)\big],
\end{align}
where  $U:[0,+\infty)\to\R$ is the so-called  utility function of the investor and it is assumed to satisfy the following assumptions:
\begin{itemize}
\item[\bf{(H4)}] 
 $U\in C^1((0,+\infty);\R)$;\\ 
$U$ is concave and  strictly increasing;\\
$\lim_{x\to+\infty} U'(x)=0.$
\end{itemize}

For any $[0,T-t]$-valued stopping time $\theta$,  $v$ satisfies the Dynamic Programming Principle (DPP) 
\begin{equation}\label{eq:DPP_v}
v(t,x)= \underset{\alpha\in \cA}\sup \;\E_{t,x}\Big[v(t+\theta,X^{\alpha}_{t+\theta})\Big],
\end{equation}
from which, at least formally, one can show that the Hamilton-Jacobi-Bellman (HJB) equation associated with the optimal control problem \eqref{eq:value_v} is
\begin{align}\label{HJB}
-v_t +\sup_{a\in A} \left( - x \left(r(t) + a^\top(b(t)-r(t)\mathbbm 1) + g(t,a)\right) v_x -\frac{1}{2}x^2 Tr[a(\sig\sig^\top)(t) a^\top] v_{xx} \right) & =   0
\end{align}
for $t\in [0,T)$, $x\geq 0$, completed with the terminal condition $v(T,x)=U(x)$ {for $x\geq 0$ (see \cite[Section 3.6.1]{Pham_book}). We refer the reader to \cite[Section 3, Chapter 4]{YZ} and the references therein for a complete overview on the dynamic programming approach to optimal control problems. \\
In the general case, $v$ is not expected to have sufficient regularity to satisfy the previous equation in the classical sense and even if \eqref{HJB} admits a classical solution, it is rarely found explicitly. To handle the problem in its full generality, the notion of viscosity solution is needed (see \cite{CIL92} for an overview). Indeed, under suitable assumptions, it can be proved (see for instance \cite[Theorems 5.2 and 6.1]{YZ}) that $v$ defined in \eqref{eq:value_v} is the unique continuous viscosity solution to \eqref{HJB} on $[0,T]\times [0,+\infty)$.

\section{ The numerical scheme}\label{sect:disc_scheme}
We consider here the scheme analyzed in \cite{PicaReisi18_prob}.  It belongs to the family of the so-called semi-Lagrangian (SL) schemes (see \cite{CF95, debrabant2013semi, Kry99, M89} for their earlier introdution)  which are based on discretization of the control set $\mathcal A$ and a Markov chain approximation of the associated optimal control problem. For completeness, we briefly discuss below the main features  of the scheme. We refer the reader to \cite{PicaReisi18_prob} for further details.

\subsection{Description of the scheme}\label{sec:descr_scheme}
We start by introducing a discretization in time.
Let $N\geq 1$,
$$
h=T/N \quad \text{and} \quad t_n=nh,
$$
for $n=0,\dots,N$.
The first step in our approximation is to introduce a time discretization of the control set. We consider  the set $\cA_h$ of controls $\alpha \in\cA$ which are constant in each interval $[t_n, t_{n+1}]$, for $n=0,\dots,N-1$, i.e.
$$
\cA{_h}:=\Big\{\alpha\in\mathcal A:  \alpha_s{(\omega)} \equiv \sum^{N-1}_{i=0} a_i \mathbbm 1_{s\in [t_i,t_{i+1})} 
{\; \forall \omega \in \Omega } \;  \text{ s.t. } a_i\in A, \quad i=0,\ldots, N-1 \Big\}.
$${
In what follows, we identify any element $\alpha\in \mathcal A_h$ by the sequence of random variables $a_i$ taking values in $A$ (denoted  by $a_i \in A$ for simplicity) and will write $\alpha\equiv (a_0,\ldots, a_{N-1})$.}
We denote by $v_h$ the value function obtained by restricting the supremum in \eqref{eq:value_v} to controls in $\cA_h$, that is 
\begin{equation}\label{eq:value_v_delta}
v_h(t,x):=\underset{\alpha\in\cA_h}\sup \E_{t,x}\left[U(X^{\alpha}_T)\right].
\end{equation}
Clearly, since $\cA_h\subseteq \cA$, one has
\be\label{eq:est_geq}
v(t,x)\geq  v_h(t,x),
\ee
for any $t\in[0,T]$, $x\geq 0$. {An upper bound of order $1/6$ for the error related to this approximation was first obtained by Krylov in \cite{Kry99}. Recently, this estimate has been improved to the order $1/4$ in \cite{JakoPicaReisi18}, so that one has 
\be\label{eq:est_krylov}
v(t,x)\leq  v_{h}(t,x)+Ch^{1/4}
\ee
for some constant $C\geq 0$. {We point out that the results in \cite{Kry99} and \cite{JakoPicaReisi18} require some additional assumptions on the coefficients and do not directly apply to problem \eqref{eq:SDE} to \eqref{eq:value_v}. It is possible that analogous estimates hold also in the setting of the present paper, but since we do not make use of \eqref{eq:est_krylov} here, we did not check this point in detail.
{Indeed, a main objective of the present paper is to by-pass the estimate \eqref{eq:est_krylov}, which turns out to be a bottleneck in the provable approximation order, while still using the piecewise constant policy approximation itself by building an approximation to $v_h$. The more important observation from \cite{JakoPicaReisi18} is therefore that a better order than 1/4 is not provable in the general case of Lipschitz viscosity solutions.
Then no matter how precise the estimates obtained for the error of the final approximation to $v_h$ are, without any further information the upper error bounds to $v$ cannot be more accurate than $O(h^{1/4})$. Section \ref{sec:dual_problem} will show how this term can be replaced by an expression which is computable from the dual problem and
provides sharper bounds in our tests (see Section \ref{sec:num}).
}
}


For any  given $\alpha\equiv(a_0,\ldots,a_{N-1})\in \mathcal A_{h}$, we consider 
the Euler-Maruyama approximation of the process $X^{t,x,\alpha}_\cdot$  given by the following recursive relation:
\begin{equation}\label{eq:def_M}
X_{t_{i+1}}= X_{t_{i}}  + h\, X_{t_{i}}  \left( r(t_i) +a_i ^\top (b(t_i) -r(t_i) \mathbbm  1) + g(t_i,a_i) \right) +  X_{t_i} a_i^\top \sig(t_i) \Delta B_i 
\end{equation}
for $i=0,\ldots, N-1$. The increments  
$\Delta B_i:= B_{t_{i+1}}- B_{t_i}$
are independent, identically distributed random variables such that
\be\label{eq:N_0_h}
\Delta B_i \sim \sqrt{h} \, \mathcal N(0,I_d)\qquad \qquad \forall i=0,\dots,N-1.
\ee
We denote  by  $\overline X^{t_n,x,\alpha}_{\cdot}$ the solution to \eqref{eq:def_M} with the control $\alpha\equiv(a_0,\ldots,a_{N-1})\in\cA_{h}$ and such that $\overline X^{t_n,x,\alpha}_{t_{n}}=x$. In the next step, we work towards a Markov chain approximation of $\overline X^{t_n,x,\alpha}_{\cdot}$.

Let us start for simplicity with the case $d=1$.
Let $M\geq 2$ and denote by $\{z_i\}_{i=1,\ldots, M}$ the zeros of the Hermite polynomial $H_{_M}$ of order $M$  and by $\{\omega_i\}_{i=1,\ldots ,M}$ the corresponding weights given by 
$$
\omega_i=\frac{2^{M-1} M!\sqrt{\pi}}{M^2 [H_{_{M-1}}(z_i)]^2},\qquad i=1,\ldots,M.
$$
With the definitions
$$
\lambda_i:=\frac{\omega_i}{\sqrt{\pi}}\quad\text{and}\quad \xi_i:=\sqrt{2}z_i, \qquad i=1,\ldots, M,
$$
one can make use of the following approximation (see, e.g., \cite[p.~395]{Hildebrand56})
\begin{align}\label{eq:GH_approx}
 \int^{+\infty}_{-\infty} f(y) \frac{e^{-\frac{y^2}{2}}}{\sqrt{2\pi}}dy
 \approx 
 \sum^M_{i=1} \lambda_i f(\xi_i),
 \end{align}
which holds  for any smooth real-valued function $f$ (say $f$ at least $C^{2M}$). 
 Observing that $\lambda_i\geq 0, \forall i=1,\ldots,M$, and $\sum^M_{i=1} \lambda_i=1$, given the  sequence $\{\zeta_n\}_{n=0,\ldots,N-1}$ of i.i.d.\ random variables  such that for any $n=0,\ldots,N-1$
$$
  \P(\zeta_n=\xi_i)=\lambda_i,\quad i=1,\dots,M,
$$ 
 one has 
  $$
 \E[\zeta_n]=0\qquad\text{and}\qquad\text{Var}[\zeta_n]=1\qquad\forall n=0,\ldots,N-1.
 $$
\begin{center}
\begin{table}
\floatbox[{\capbeside\thisfloatsetup{capbesideposition={right,top},capbesidewidth=11.5cm}}]{figure}[\FBwidth]
{\caption{Analytical expressions of  $\{(\xi_i,\lambda_i)\}_{i=1,\ldots,M}$ for $M=2,3,4$. We refer to \cite[p.\ 464]{Beyer87} for numerical approximations of $\{(z_i,\omega_i)\}_{i=1,\ldots,M}$ for larger $M$.}\label{tab:lambda_xi}}
{\begin{tabular}{|c||c|c|}
\hline
 &  $\xi_i$ & $\lambda_i $\\
\hline  
\hline
$M=2$  & $\pm 1$ &  $1/2$ \\ 
\hline 
$M=3$  & 0 & $2/3$ \\ 
& $\pm\sqrt{3}$ & $1/6$\\
\hline
$M=4$  & $\pm \sqrt{3-\sqrt{6}}$  &  $(3+\sqrt{6})/12$  \\ 
& $\pm \sqrt{3+\sqrt{6}}$  &   $(3-\sqrt{6})/12$ \\
\hline
\end{tabular}
\hspace{-1.5 cm}
}
\end{table}
\end{center}
For any control $\alpha\equiv(a_0,\ldots,a_{N-1})\in\cA_{h}$, we will denote by $\widehat X^{t_n,x,\alpha}_\cdot$ 
 the Markov chain approximation of the process $\overline X^{t_n,x,\alpha}_\cdot$ , i.e.
 \be\label{eq:def_Xh}
 \begin{cases}
 \widehat  X_{t_n} \hspace{-0.3cm}& =x,\\
\widehat  X_{t_{i+1} } \hspace{-0.3cm} & =\widehat X_{t_{i}}  + h\, \widehat X_{t_{i}}  \left( r(t_i) +a_i  (b(t_i) -r(t_i) ) + g(t_i,a_i) \right) + \sqrt{h}\,\widehat  X_{t_i} a_i \sig(t_i) \,\zeta_i,  
  \end{cases}
 \ee
 for $i=n,\ldots,N-1$.

Applying to \eqref{eq:DPP_v} with $\theta = h$ the piecewise control approximation, the Euler-Maruyama discretization and the Gau{\ss}-Hermite quadrature formula \eqref{eq:GH_approx}, we obtain the following recursive semidiscrete approximation (i.e., discrete in time, continuous in space) of the value function 
 \small
 \begin{equation}
 \label{eq:SL_V}
\begin{cases}
 V(t_n,x)&=\underset{a\in A}\sup\; \sum^M_{i=1}\lambda_{i} V\left(t_{n+1},x  + h\, x \left( r(t_n) +a (b(t_n) -r(t_n) ) + g(t_n,a) \right) + \sqrt{h}\,x\, a \sig(t_n) \,\xi_i\right)\\
  &=\underset{a\in A}\sup\;\E_{t_n,x}\left[ V(t_{n+1},\widehat X^{a}_{t_{n+1}})\right], \hfill n=N-1,\ldots,0, \\
 V(t_{_N},x)&= U(x).
\end{cases}
 \end{equation}
 \normalsize
Iterating gives the following representation formula for $V$:
$$
V(t_n,x)=\underset{\alpha\in\mathcal A_{h}}{\sup}\E_{t_n,x}\left[U(\widehat X^{\alpha}_T)\right].
$$
\normalsize

In the case of $d>1$, it is possible to extend  formula \eqref{eq:GH_approx} by a tensor product approximation as discussed in \cite{PicaReisi18_prob}.

We introduce now a discretization of the space variable. 
Let ${\Delta x>0}$ and consider the space grid  $\mathcal G_{\Delta x}:=\{x_m = m \Delta x : m\in{\Z}\}$. We also write $\mathcal G^+_{\Delta x}:=\{x_m = m \Delta x : m\in{\N}\}$.
Let $\mathcal I[\cdot]$ denote the linear interpolation operator {with respect to the space variable}, 
satisfying for every Lipschitz function $\phi$ (with Lipschitz constant $L_\phi$):
\be\label{eq:interp}
\begin{cases}
\mbox{$(i)$   $\mathcal I[\phi](x_m)=\phi(x_m)$, $\forall m\in\Z$,}\\
\mbox{$(ii)$   $|\mathcal I[\phi](x)-\phi(x)| \leq L_\phi \dx$,}\\
\mbox{$(iii)$  $|\mathcal I[\phi](x)-\phi(x)| \leq C \dx^2 \|D^{2} \phi\|_\infty$ \ if  $\phi\in C^2(\R)$,}\\
\mbox{$(iv)$ for any functions $\phi_1,\phi_2:\R\conv\R$, $\phi_1\leq \phi_2$ $\Rightarrow$ $\mathcal I[\phi_1]\leq \mathcal I[\phi_2]$.}
\end{cases}
\ee

We define an approximation $\VF$ on this fixed grid as follows:
 \begin{equation}
 \label{eq:SL_V_fully}
\hspace{-0.3cm}\begin{cases}
 \VF(t_n,x_m)= &\hspace{-0.3cm}\underset{a\in A}\sup\;\sum^{{ M}}_{i=1} { \lambda_i}\; \mathcal I[\VF]\Big(t_{n+1},x_m  + h\, x_m \left( r(t_n) +a^\top (b(t_n) -r(t_n) \mathbbm  1) + g(t_n,a) \right)  \\
 & \hspace{4.2 cm}+ \sqrt{h}\,x_m\, a^\top \sig(t_n) \, \xi_i\Big),\\
 \VF(t_{_N},x_m)=& \hspace{-0.3cm}U(x_m),
\end{cases}
 \end{equation}
 for $n=N-1,\ldots,0$ and $m\in \N$.
{We will refer to this as the fully discrete scheme.}
For $M=2$, the scheme coincides with the one introduced by Camilli and Falcone in \cite{CF95}. However, as explained in the next section, the error estimates derived in the present paper improve the state of the art for this class of schemes only when $M>2$ is considered.
 
  

\subsection{An \textit{a priori} lower bound for $v$}

Under suitable assumptions,  \textit{a priori} estimates of the following form are proven in \cite{PicaReisi18_prob}:
\be\label{eq:error_asym}
\begin{split}
\hspace{-0.5cm} -C \left(h^{(M-1)/2M} + \frac{\Delta x}{h}\right)\leq v(t_n,x_m) - \VF(t_n,x_m) \leq   C \left(h^{1/4} +  h^{(M-1)/2M} +\frac{\Delta x}{h}\right)
\end{split}
\ee
 for any $n=0,\ldots, N$, $m\in \N$ and  a constant $C$, possibly depending on $x_m$
(the dependency of $C$ on $x_m$ can be explicitly derived and one has $C\leq C_0 (1+x_m^{2M})$ for some constant $C_0$).
{In particular, we make for now the assumption that $U$ is Lipschitz continuous.
We will discuss how to obtain bounds for some non-Lipschitz $U$ in subsection \ref{subsec:non-Lip}.}

The \textit{a priori} bounds are obtained by a direct comparison between the optimisation problems \eqref{eq:value_v_delta} and \eqref{eq:SL_V_fully}.
Contributing to the error estimates above 
are:
the Euler-Maruyama error of order $h^{1/2}$; the Gau\ss-Hermite quadrature error of order $h^{M-1}$; and the interpolation error of {order $\Delta x$,
accumulated over $1/h$ steps to $\Delta x/h$}. The bounds \eqref{eq:error_asym} are then the result of the use of \eqref{eq:est_geq} and \eqref{eq:est_krylov}  for the piecewise constant controls approximation, which introduces the aforementioned asymmetry in the estimates (given by the term $h^{1/4}$ in the right-hand side of \eqref{eq:error_asym}), and  of a regularization procedure, the so-called ``shaking coefficients'' technique in \cite{K97} and subsequent works, which is the classical tool to deal with nonsmooth solutions.


Adapting the arguments of \cite{PicaReisi18_prob} to the present problem, we can obtain the following \textit{a priori} estimate for the lower bound of the quantities $v-V$ and  $v-\VF$:
\begin{prop} \label{prop:lower}
Let  assumptions (H1) to (H3) be satisfied { and let the function $U$ be Lipschitz continuous with Lipschitz constant $L\geq 0$} . Then, there exists a constant $C\geq 0$ such that for any $n=0,\ldots,N$, $x\geq 0$
\be\label{eq:est_lower_first}
v(t_n,x)
\geq V(t_n,x) - L C (1+x^{2M}) h^{(M-1)/2M}, 
\ee
and for any  $n=0,\ldots,N$, $m\in \N$
\be\label{eq:est_lower_fully}
 v(t_n,x_m)
\geq  \VF(t_n,x_m) 
 - L C (1+x_m^{2M}) \left(h^{(M-1)/2M}  + \Delta x/h \right).
\ee
\end{prop}
\begin{proof}
When only the time discretization is taken into account, the estimate  \eqref{eq:est_lower_first} directly follows by \cite[Section 4.2, equation (4.9)]{PicaReisi18_prob}.
%
Moreover by  \cite[Section 4.3]{PicaReisi18_prob} for the fully discrete scheme one has (\ref{eq:est_lower_fully}),
{where $C$ only depends on $M$, $T$, the constants $K_0$ and $K_1$ in assumption (H2)}.
\end{proof}

{
\begin{remark}
Balancing the terms  $h^{(M-1)/2M}$ and $\Delta x/h$ on the right-hand side of (\ref{eq:est_lower_fully}) by judicious choice
of $\Delta x$ in relation to $h$ leads to 
$$
 v(t_n,x_m) - \VF(t_n,x_m)
\geq   - L C (1+x_m^{2M}) \left(h^{(M-1)/2M}  + \Delta x^{(M-1)/(3M-1)}\right).
$$
\end{remark}
}

The scheme we are considering is monotone, stable and it has order one of consistency (for smooth test functions) for any $M\geq 2$.
For a scheme of this type, (upper and lower) error bounds of order ${1/4}$ in $h$ have been provided in \cite{BJ07, DJ12}  by PDE techniques.
Splitting each contribution to the error, namely the control discretization and the Euler-Maruyama and Gau\ss-Hermite approximations, the probabilistic proof proposed in \cite{PicaReisi18_prob} gives an improvement to the lower bound of these estimates by increasing the value of $M>2$, i.e.\ by using a more accurate quadrature formula.
{For large $M$, the order is arbitrarily close to 1/2 in $h$ and $1/3$ in $\Delta x$, improving the corresponding orders $1/4$ and $1/5$ obtained for $M=2$.
It is hence for $M>2$ that the term $h^{1/4}$ in the upper bound  becomes strictly dominant and restricts the order to $1/4$, independent of $M$.}
{
The analysis that follows aims to eliminate 
this dominant term and replace it by a term which can be computed \textit{a posteriori} from the numerical solution of the original and its dual problem, which
we expect to be smaller generally than that from the 
Gau\ss-Hermite approximations.
This is confirmed in our tests. Hence we 
provide a computable upper bound which is empirically of the same order as the lower bounds obtained in Proposition \ref{prop:lower}.
}

\section{Duality-based error estimates}\label{sec:dual_problem}
In this  section we discuss how duality theory can be employed to obtain an upper bound of the error associated with our approximation scheme. Assuming to be able to extend either the PDE-based error estimates in \cite{K97, K00, BJ02, BJ05, BJ07} or the probabilistic ones in \cite{PicaReisi18_prob} to the particular problem \eqref{eq:SDE} and \eqref{eq:value_v}, this would result in both cases in an upper bound of order $1/4$ in $h$ for any choice of $M\geq 2$, as explained at the end of the previous section.
{
We show here that for our class of problems it is possible to pass through the definition of a dual problem to 
replace these \textit{a priori} estimates by \textit{a posteriori} computable bounds, which are empirically significantly smaller.
}

\subsection{The dual problem}
The dual problem associated with  \eqref{eq:SDE} and \eqref{eq:value_v}  is defined in \cite{CuoLiu00} by 
\begin{equation}\label{SDEd_CL}
 \!\! \left\{\begin{array}{l}
 \!\! \mathrm{d}Y_s=-\big(r(s)\!+\!\tilde g(s,\nu_s)\big) Y_s\, \mathrm{d}s+(\sig\sig^T(s))^{-1} Y_s(r(s)\mathbbm 1 \!-\! b(s)\!-\!\nu_s)\cdot \sig(s) \,\mathrm{d}B_s, \!\quad s\in(t, T), \\
 \!\! Y_t=y, \end{array}\right.
 \end{equation}
 for all $t\in [0,T)$, where
 \vspace{-0.4 cm}
 \be\label{eq:def_gtilde}
\widetilde g(t,\nu):=\sup_{a\in A}\big\{g(t,a)-a\cdot \nu\big\},
\ee 
and the dual utility function $\widetilde U$ is the convex conjugate of $U$, i.e.
$$
\widetilde U(y):=\underset{x\geq  0}\sup\{U(x)-xy\}.
$$ 
The dual value function is defined by 
\be\label{eq:value_tv}
\tilde v(t,y) = \inf_{\nu\in \mathcal V} \mathbb E_{t,y}\big[ \widetilde U(Y^\nu_T)\big], 
\ee
where $\mathcal V$ is the set of $\R^d$-valued  progressively measurable processes such that $\int^T_0 |\nu_s|^2 \, \mathrm{d} s + \int^T_0 \widetilde g(s,\nu_s) \, \mathrm{d} s < +\infty$.
One has the following duality result:
\begin{prop}[\cite{CuoLiu00}, Theorem 2]\label{prop:duality_cont}
Let assumptions (H1) to (H4) be satisfied. Then for any $t\in[0,T]$,  $x\geq 0$, the primal and dual value functions, $v$ and $\widetilde v$,
satisfy the conjugate relation
\be\label{eq:duality_CL}
v(t,x)=\underset{y>0}\inf\big\{\widetilde v(t,y)+xy\big\}.
\ee
\end{prop}
\begin{remark}
The results in \cite{CuoLiu00} hold also if $r,b,\sig$ and $g$ are stochastic processes. However, as our approximation scheme makes use of the Markovian structure,
{we would have to add extra variables to the state space to account for this, which is outside the scope of this work.}
\end{remark}
\begin{remark}\label{rem:on_U}
As mentioned in \cite[Section 6.5]{Rogers}, the usual Inada condition 
$$
\lim_{x\to 0^+} U'(x)=+\infty
$$
(requested in \cite{CuoLiu00}) is not necessary for proving the main duality results. 
\end{remark}

\subsection{Approximation of the dual problem}\label{sec:approx_dual}
The scheme presented  in Section \ref{sec:descr_scheme} can be used to approximate the value function $\widetilde v$ associated with the dual problem {\eqref{SDEd_CL}-\eqref{eq:value_tv}}.
{To this end, we define by $\Gamma\subset\R^d$ a compact set 
and by $\mathcal V^\Gamma \subset \mathcal V$ the set of all a.s.\ $\Gamma$-valued elements  of $\mathcal V$.
One clearly has
\be\label{eq:value_tv_gam}
\tilde v{(t,y)} \le \tilde v^\Gamma(t,y) = \inf_{\nu\in \mathcal V^\Gamma} \mathbb E_{t,y}\big[ \widetilde U(Y^\nu_T)\big].
\ee
{ If there exists a uniformly bounded optimal control $\nu^*\in \mathcal V$,   
 one can find a compact set $\Gamma$ such that 
 $\tilde v = \tilde v^\Gamma$.
 Otherwise, such an approximation introduced on the set of controls will result in a strictly bigger value function and in a 
duality gap  which does not diminish under mesh refinement and can only be decreased by increasing $\Gamma$.
Nonetheless, the inequalities stated in this section still hold in this case.}

Let further $ \zeta_i$, $i=0,\ldots,N-1$, be i.i.d.\ copies of the increments from the definition of the primal approximation.
}
For the discrete time scheme, one can then recursively define
 \begin{equation}
 \label{eq:SL_tV}
\begin{cases}
 \widetilde V(t_n,y)= & \underset{\gamma\in  \Gamma}{\inf}\;\E_{t_n,y}\big[\widetilde V(t_{n+1},\widehat Y^{\gamma}_{t_{n+1}})\big],\qquad\qquad n=N-1,\ldots,0, \\
 \widetilde V(t_{_N},y)=& \widetilde U(y) ,
 \end{cases}
 \end{equation}
 where $y\geq 0$, 
  and $\widehat Y^{t_n,y,\gamma}_\cdot$ is the Markov chain recursively defined by  
 \be\label{eq:def_tildeYh}
\begin{cases}
 \widehat Y_{t_n}=y,\\
 \widehat Y_{t_{i+1}}= \widehat Y_{t_{i}} - h \big(r(t_i)\!+\!\tilde g(t_i,\gamma_i)\big) \widehat Y_{t_i}+\sqrt{h}(\sig\sig^T(t_i))^{-1} \widehat Y_{t_i}(r(t_i)\mathbbm 1 \!-\! b(t_i)\!-\!\gamma_i)\cdot \sig(t_i) \zeta_i
 \end{cases}
 \ee
for $i=n,\ldots,N-1$. Denoting by
$\mathcal V_h^\Gamma$  the set of all
 $\nu \equiv(\gamma_0,\ldots,\gamma_{N-1})$ adapted to the filtration generated by
 $(\zeta_0,\ldots,\zeta_{N-1})$, 
 with $\gamma_i$ random variables taking values in $\Gamma$, $i=0,\ldots,N-1$ one has 
 $$
 \widetilde V(t_n,y)=  \underset{\nu\in \mathcal V_h^\Gamma}{\inf}\;\E_{t_n,y}\big[\widetilde U(\widehat Y^{\nu}_{T})\big].
 $$
}

The fully discrete version of the scheme is then given by 
\begin{equation}
 \label{eq:SL_tV_fully}
\begin{cases}
 \tVF(t_n,y_j)= &
 \underset{\gamma \in \Gamma}\inf\;\sum^{{ M}}_{i=1} {\lambda_i}\; \mathcal I[\tVF]\left(t_{n+1},y_j  + h\, y_j \left( r(t_n)\!+\!\tilde g(t_n,\gamma) \right)\right. \\
 & \hspace{4.2 cm} \left. + \sqrt{h} y_j (\sig\sig^T(t_n))^{-1} (r(t_n)\mathbbm 1 \!-\! b(t_n)\!-\!\gamma)\cdot \sig(t_n) \xi_i\right), \\
 \tVF(t_{_N},y_j)=& \hspace{-0.3cm}\widetilde U(y_j),
\end{cases}
 \end{equation}
 for $n=N-1,\ldots,0$ and $j\in \N$.

%

\subsection{An \textit{a priori} upper bound for $\tilde v$}\label{subsect:upper} 
The approximation scheme we defined for the dual problem is the same we used for the primal one, with the only difference that we have to handle a minimization problem. Therefore, we can  use the arguments in \cite{PicaReisi18_prob} to obtain  an accurate upper bound for the differences $\widetilde v-\widetilde V$ and $\widetilde v-\tVF$.

\begin{prop} \label{prop:upper}
Let  assumptions (${\text{H1}}$) to (${\text{H4}}$) be satisfied {and let $\widetilde U$ be Lipschitz continuous with Lipschitz constant $\widetilde L\geq 0$}. Then, there exists a constant ${\widetilde C}\geq 0$,
such that for any $n=0,\ldots,N$, $y>0$,
\be\label{eq:est_upper_first}
\widetilde v(t_n,y)
\leq \widetilde V(t_n,y)+ \widetilde L {\widetilde C} (1+y^{2M}) h^{(M-1)/2M}, 
\ee
and for any  $n=0,\ldots,N$, $j\in \N^+$,
\be\label{eq:est_upper_fully}
\widetilde v(t_n,y_j)
\leq \tVF (t_n,y_j) + \widetilde L {\widetilde C} (1+y_j^{2M}) \left(h^{(M-1)/2M}  + 
\Delta x/h \right).
\ee
\end{prop}
\begin{proof}
{
The result follows by applying the estimates from Section 4.2 in \cite{PicaReisi18_prob}, adapted to a minimisation problem, to $\tilde v^\Gamma$ in (\ref{eq:value_tv_gam}).
Under the assumptions (H1)-(H3), one has 
\begin{align*}
 \|(\sig\sig^T(t))^{-1}(r(t)\mathbbm 1-b(t)-\gamma)-& (\sig\sig^T(s))^{-1}(r(s)\mathbbm 1-b(s)-\gamma)\|\\
& + |\tilde g(t,\gamma)-\tilde g(s,\gamma)|+|r(t)-r(s)|\leq C_0|t-s|^{1/2},
\end{align*} 
where $C_0$ depends  on $T$, the constants $K_0, K_1$ and $\eta$ in assumptions (H2)-(H3) and the uniform bounds on the elements of $\Gamma$, so that the dynamics \eqref{SDEd_CL} satisfies the assumptions in \cite{PicaReisi18_prob}. 
}
\end{proof}

{
\begin{remark}\label{rem:controls}
A similar truncation strategy as for $\mathcal V$ can be applied to {the set of controls } $\mathcal A$ if $A$ is unbounded. For the thus obtained numerical solution, the inequalities
in Proposition \ref{prop:lower} still hold. Again, in this case, the duality gap can only be reduced by increasing $A$ and not {only}  by letting $h$ and $\Delta x$ go to 0 alone.
\end{remark}
}

%

\subsection{Using duality in error estimates}\label{sect:idea}

In the sequel, we will use the following notation: for any $n=0,\ldots,N$, $x>0$ 
 \begin{align}
G^{h}(t_n, x)& :=\underset{y>0}{\min} \Big\{\widetilde V(t_n,y)+x y \Big\}-V(t_n,x) \label{eqn:gap} \\
I^h(t_n, x) & :=\arg\min_{y>0} \Big\{\widetilde V (t_n,y)+x y\Big\}, \nonumber 
\end{align}
and for any $n=0,\ldots,N$, {$m\in\N^+$}
 \begin{align}
G^{h,\Delta x}(t_n,x_m)& :=\underset{j\in \N^+}{\min} \Big\{\tVF(t_n,y_j)+x_m y_j \Big\}-\VF(t_n,x_m), \label{eqn:gap_fully}\\
I^{h,\Delta x}(t_n,x_m) & :=\arg\min_{j\in \N^+} \Big\{\tVF(t_n,y_j)+x_m y_j\Big\}. \nonumber
\end{align}
We refer to $G^h$ and $G^{h,\Delta x}$ as the \textit{numerical duality gap} of the semidiscrete and fully discrete scheme respectively.  


One has the following result:

\begin{teo}\label{eq:error}
Let assumptions (H1) to (H4) be satisfied {and let $U$ and $\widetilde U$ be Lipschitz continuous with Lipschitz constants $L$ and $\widetilde L$, respectively}. Then, there exist some constants $C, \widetilde C\geq 0$ 
such that for any $n=0,\ldots, N$, $x> 0$
\begin{align}
\label{eq:bounds}
-L C (1+x^{2M}) h^{(M-1)/2M}\leq v(t_n,x)- V(t_n,x) \leq G^{h}(t_n,x) +\widetilde L {\widetilde C} (1+(I^{h}(t_n,x))^{2M}) h^{(M-1)/2M} 
\end{align}
and for any $n=0,\ldots, N$, $m\in\N^{ +}$
\begin{equation}\label{eq:error_fully}
\begin{split}
-L C (1+x_m^{2M})\Big(h^{(M-1)/2M}+  \Delta x/h \Big) \leq v(t_n,x_m)-\VF(t_n,x_m)
\\ \leq G^{h,\Delta x}(t_n,x_m) +\widetilde L {\widetilde C} (1+(I^{h,\Delta x}(t_n,x_m))^{2M})\Big(h^{(M-1)/2M}+  \Delta x/h \Big).
\end{split}
\end{equation}
\end{teo}

\begin{proof}
The first inequalities in \eqref{eq:bounds} and \eqref{eq:error_fully} follow directly by Proposition \ref{prop:lower}. It remains to prove the upper bounds. We prove the result for the semi-discrete scheme, while the proof for the fully discrete scheme follows by similar arguments.
Thanks to Proposition \ref{prop:duality_cont}, Proposition \ref{prop:upper} and the definition of $I^h(\cdot,\cdot)$ one has 
\begin{align*}
v(t_n,x) & =\inf_{y>0}\left\{ \widetilde v(t_n,y) + xy\right\}\leq \inf_{y>0}\left\{ \widetilde V(t_n,y) + xy + \widetilde L \widetilde C (1 + y^{2M}) h^{(M-1)/2M}\right\}\\
& \leq \widetilde V(t_n,I^h(t_n,x)) + x \, I^h(t_n,x) + \widetilde L \widetilde C \left(1 + (I^h(t_n,x))^{2M}\right)h^{(M-1)/2M}\\ 
& = \inf_{y>0}\left\{ \widetilde V(t_n,y) + xy \right\}  +  \widetilde L \widetilde C \left(1 + (I^h(t_n,x))^{2M}\right)h^{(M-1)/2M}.
\end{align*}
Therefore, 
\begin{align*}
v(t_n,x) - V(t_n,x) \leq  \inf_{y>0}\left\{ \widetilde V(t_n,y) + xy \right\} -  V(t_n,x) +  \widetilde L \widetilde C \left(1 + (I^h(t_n,x))^{2M}\right)h^{(M-1)/2M},
\end{align*}
which gives the desired result.
\end{proof}

Observe that due to the particular convexity feature of the dual problem, the quantity $I(x)$ typically increases as $x$ approaches $0$.

The duality gap for the fully discrete scheme is computable efficiently, see e.g. \cite[Section 3.4]{FerrFalcBook},
so that \eqref{eq:error_fully} provides a practically useful \textit{a posteriori} bound. 

\textit{A priori} bounds could be obtained by proving that the numerical duality gap $G^{h}$  (resp.  $G^{h,\Delta x}$) decays with order at most $h^{(M-1)/2M}$ (resp. $h^{(M-1)/2M} + Dx/h$). This requires a proof that $V$ and $\widetilde V$ (resp. $\VF$ and $\tVF$) satisfy an approximated duality relation. 
{Indeed, the key feature of dynamics \eqref{eq:SDE} and \eqref{SDEd_CL} leading to the conjugate relation \eqref{eq:duality_CL} is the following so called ``polar property''
$$
\underset{\nu\in\mathcal V}\sup\, \E\left[ X^{t,x,\alpha}_T Y^{t,y,\nu}_T\right] = xy \qquad \forall x,y \geq 0, t\in [0, T], \alpha\in \mathcal A.
$$
{
For the discrete time dynamics $\widehat X_\cdot$ and  $\widehat Y_\cdot$ defined in \eqref{eq:def_Xh} and \eqref{eq:def_tildeYh}, respectively, a straightforward 
calculation  shows that  for any $\alpha\equiv(a_{n}, \ldots, a_{N-1})\in\mathcal A_{h}$ and $\nu\equiv(\gamma_{n}, \ldots, \gamma_{N-1})\in\mathcal V^{\Gamma}_{h}$
\small
\begin{align*}
& \widehat X^{t_n,x,\alpha}_T\widehat  Y^{t_n,y,\nu}_T-xy\nonumber \\
& =\sum^{N-1}_{i=n}  \widehat X^{t_n,x,\alpha}_{t_i} \widehat  Y^{t_n,y,\nu}_{t_i} \Big\{h\Big(a_i(b(t_i)-r(t_i))+g(t_i,a_i)-\tilde g(t_i,\gamma_i)+a_i(r(t_i)-b(t_i)-\gamma_i)\zeta^2_i\Big)\\
& \qquad+h^2\Big(-\big(r(t_i)+\tilde g(t_i,\gamma_i)\big)\big(r(t_i)+a_i(b(t_i)-r(t_i))+g(t_i,a_i)\big)\Big)\\
& \qquad+(\ldots)\zeta_i+(\ldots ) (\zeta^2_i-1)\Big\}.\nonumber
\end{align*}
\normalsize
Taking the expectation in the expression above, thanks to the independence and distribution of the random variables $\zeta_i$ and the definition of the convex conjugate $\tilde g$, one  gets 
\beno
\begin{aligned}
& \E\Big[ \widehat X^{t_n,x,\alpha}_T\widehat  Y^{t_n,y,\nu}_T\Big]-xy\leq C h^2 \sum^{N-1}_{i=n} \E\Big[\;\widehat X^{t_n,x,\alpha}_{t_i} \widehat  Y^{t_n,y,\nu}_{t_i}\;\Big]
\end{aligned}  
\eeno
\normalsize
for some constant $C$ depending on $T$, the uniform bounds on $A$ and $\Gamma$ and  the constants appearing in assumption  (H2). For any $i=n,\ldots,N-1$ one can easily prove that 
$$
\E\Big[\widehat X^{t_n,x,\alpha}_{t_i} \widehat  Y^{t_n,y,\nu}_{t_i}\Big] \leq xy\, C e^{CT},
$$
for some possibly different constant $C\geq 0$, so that it is possible to conclude that there exists some $C\geq 0$ such that 
$$
\underset{\nu\in\mathcal V^\Gamma_h}\sup\, \E\left[ \widehat X^{t_n,x,\alpha}_T \widehat Y^{t_n,y,\nu}_T\right] \leq  xy\left(1+Ch\right) \qquad \forall x,y \geq 0, n=0,\ldots,N,  \, \alpha\in \mathcal A.
$$
}

We conjecture that a similar approximate lower bound also holds. This finds a confirmation in our numerical tests (see Tables \ref{tab:test2Local} and \ref{tab:test2GlobalN54} in Section \ref{sec:num}) where at least first order of convergence in $h$  of the numerical duality gap is observed. However the rigorous prove of the result involves delicate convex analysis arguments 
and we plan to investigate this point in future work.

{
\subsection{The case of non Lipschitz utility functions}
\label{subsec:non-Lip}
We assumed for the results above that the primal (and, where applicable, dual) utility functions $U$ (and $\widetilde U$) are Lipschitz continuous
(see Propositions \ref{prop:lower}, \ref{prop:upper}, and Theorem \ref{eq:error}).
This is a standard assumption in the numerical literature, including our previous work  \cite{PicaReisi18_prob} which we draw on here.
This property is, however, not satisfied by commonly used utility functions in finance, such as the power utility $U(x)= x^p/p$, $x\ge0$, with $p\in (0,1)$,
or the dual of the exponential utility.
To deal with such cases, we introduce a further approximation of the problem and consequently have to estimate an additional error contribution.

We assume first, in addition to (H4), that $U$ is bounded (from below) at $0$.
Letting $\rho,c_0>0$ and $x_\rho=c_0/\rho$, $y_\rho=\rho$, we define 
\begin{eqnarray}
\label{LipU}
U_\rho(x):=\left\{\begin{array}{ll}
U(0)+\frac{ U(x_\rho)-U(0)}{x_\rho}x& \text{if }\quad 0 \leq x \le x_\rho, \\
U(x) & \text{if }\quad x_\rho < x \leq y_\rho, \\
 U(y_\rho) & \text{if }\quad x > y_\rho,
\end{array}
\right.
\end{eqnarray}
so that $U_\rho$ is Lipschitz with Lipschitz constant $L_\rho :=(U(x_\rho)-U(0))/x_\rho$ and $U_\rho \to U$ as $\rho \to +\infty$ (uniformly on compact sets).

We denote by $v_\rho$, $V_\rho$ and $\VF_\rho$  the value function and the numerical approximations defined respectively by \eqref{eq:value_v},  \eqref{eq:SL_V} and \eqref{eq:SL_V_fully}, replacing $U$ with $U_\rho$.
Observe that as $U$ is concave and therefore $U_\rho \le U$, one has for any $t\in [0,T], x\geq 0$
\be\label{eq:ineqRHO}
v_\rho(t,x)\leq v(t,x).
\ee
 
Let $\widetilde U_\rho$ be the convex conjugate of the approximated utility function $U_\rho$, i.e.
$$
\widetilde U_\rho(y) :=\sup_{x\geq 0}\;\{U_\rho(x) - xy\}.
$$
We denote by $\widetilde v_\rho$, $\widetilde V_\rho$ and ${\tVF}_\rho$ the value function and the numerical approximations obtained respectively by \eqref{eq:value_tv}, \eqref{eq:SL_tV} and \eqref{eq:SL_tV_fully}, replacing $\widetilde U$ with $\widetilde U_\rho$.
Observe that $\widetilde U_\rho:[0,+\infty)\to\R$ is decreasing and Lipschitz continuous with constant $\widetilde L_\rho :=y_\rho$. Moreover, it follows by the very definition of $U_\rho$ that $\widetilde U_\rho (y)=0$ for $y\geq L_\rho$.

\begin{remark}
The modified utility function $U_\rho$ is not of class $C^1$, however the discussion in \cite[Section 6.5]{Rogers} can be used to show that \eqref{eq:duality_CL} also holds for $v$ and $\widetilde v$ replaced by $v_\rho$ and $\widetilde v_\rho$, i.e.
\be\label{eq:duality_CL_rho}
v_\rho(t,x)=\underset{y>0}\inf\big\{\widetilde v_\rho(t,y)+xy\big\}.
\ee
\end{remark}

The following large deviations-type argument is needed to estimate the error of this Lipschitz continuous approximation.

\begin{lem} \label{lem:large_dev}
Consider an $\mathbb{R}$-valued process $p$ and  an $\mathbb{R}^d$-valued process $q$, both progressively measurable 
with $\int_0^T |p_s| \, ds \le \mu T$ and $\int_0^T |q_s|^2 \, ds \le \gamma^2 T$ a.s., respectively, for some constants $\mu, \gamma \ge 0$, and let
\[
X_t = x \exp\left(\int_0^t p_s \, ds + \int_0^t q_s \, dW_s\right) 
\]
for $t\in [0,T]$.
Then
\begin{eqnarray}\label{eq:est_ld1}
\mathbb{P}\left[X_t \ge \rho \right] &\le& 2 \exp\left( - \frac{3}{8 \gamma^2 T} (\log \rho/x - \mu T)^2 \right), \\
\mathbb{P}\left[X_t \le c_0/\rho \right] &\le& 2 \exp\left( - \frac{3}{8 \gamma^2 T} (\log \rho/(c_0 x) - \mu T)^2 \right).
\label{eq:est_ld2}
\end{eqnarray}
Moreover, for each $p>0$, $x>0$ there exists $C>0$ such that
\begin{eqnarray}
\label{eq:est_ld3}
\mathbb{E}\left[X_t \mathbbm{1}_ {\{X_t \ge \rho\}} \right] 
\le C \, \rho^{-p}
\end{eqnarray}
for all $t\in [0,T]$.
\end{lem}
\begin{proof}
We have, for any $\lambda>0$,
\begin{eqnarray*}
\mathbb{P}\left[X_t \ge \rho \right] &=& \mathbb{P}\left[\exp\left(\frac{\lambda}{2}  \left(\log X_t/x -  \int_0^t p_s \, ds\right)^2\right) \ge 
\exp\left(\frac{\lambda}{2} \left(\log \rho/x - \int_0^t p_s \, ds\right)^2 \right)\right] \\
&\le&
 \mathbb{P}\left[\exp\left(\frac{\lambda}{2}   \left(\log X_t/x -  \int_0^t p_s \, ds\right)^2\right) \ge 
\exp\left(\frac{\lambda}{2} (\log \rho/x - \mu T)^2 \right)\right].
\end{eqnarray*}
Following the same steps as in the proof of Lemma 2.6 in \cite{hu2018existence}, we obtain for $\lambda \gamma^2 T < 1$
\[
\mathbb{E}\left[\exp\left(\frac{\lambda}{2}  \left( \int_0^t q_s \, dW_s \right)^2\right)\right] \le \frac{1}{\sqrt{1- \lambda \gamma^2 T }},
\]
and hence from Markov's inequality
\[
\mathbb{P}\left[X_t \ge \rho \right] \le \frac{\exp\left(-\frac{\lambda}{2} (\log \rho/x - \mu T)^2 \right)}{\sqrt{1-\lambda \gamma^2 T}}.
\]
Choosing $\lambda = 3/(4 \gamma^2 T)$ we obtain (\ref{eq:est_ld1}).

The second statement (\ref{eq:est_ld2}) follows immediately by replacing $X_t$ by $1/X_t$, $x$ by $1/x$, $(p,q)$ by $(-p,-q)$, and $\rho$ by $\rho/c_0$.

Finally, the last estimate is obtained from
\[
\mathbb{E}\left[X_t \mathbbm{1}_ {\{X_t \ge \rho\}} \right] 
\le
\sum_{k=\lfloor \rho\rfloor}^\infty (k+1) \mathbb{P}(X_t \in [k,k+1)) 
=
(\lfloor \rho\rfloor + 1)   \mathbb{P}(X_t \ge \lfloor \rho\rfloor ) +
\sum_{k=\lfloor \rho\rfloor}^\infty \mathbb{P}(X_t \ge k) ,
\]
and estimating each term by substituting $\lfloor \rho\rfloor$ and $k$ into (\ref{eq:est_ld1}).
\end{proof}


Let $G^h_{\rho}, I^h_\rho, G^{h,\Delta x}_{\rho}, I^{h,\Delta x}_{\rho}$ denote the quantities defined by  \eqref{eqn:gap} and \eqref{eqn:gap_fully} replacing $V, \widetilde V, W, \widetilde W$ by $V_\rho, \widetilde V_\rho, W_\rho, \widetilde W_\rho$. We then obtain the following extension of Theorem \ref{eq:error} to the general case of non Lipschitz utility functions.

\begin{teo}\label{eq:errorRHO}
Let assumptions (H1) to (H4) be satisfied. Then, there exist some constants $C, \widetilde C\geq 0$ 
{and $\delta: \mathbb{R}^+\!\! \rightarrow \mathbb{R}^+$\! with $\delta(x,\rho) = o(\rho^{-p})$ for all $x$ as $\rho\rightarrow\infty$ for all $p>0$},
such that for any $n=0,\ldots, N$, {$x>0$}
\begin{align}
\label{eq:boundsRHO}
\begin{split}
-L_\rho C (1+x^{2M}) h^{(M-1)/2M}\leq v(t_n,x)- V_\rho(t_n,x) \\
\leq G^{h}_\rho(t_n,x) +\widetilde L_\rho {\widetilde C} (1+(I^{h}_\rho(t_n,x))^{2M}) h^{(M-1)/2M} {+\delta(x,\rho)}
\end{split}
\end{align}
and for any $n=0,\ldots, N$, $m\in\N^{+}$
\begin{equation}\label{eq:error_fullyRHO}
\begin{split}
-L_\rho C (1+x_m^{2M})\Big(h^{(M-1)/2M}+  \Delta x/h \Big) \leq v(t_n,x_m)-\VF_\rho(t_n,x_m)
\\ \leq G^{h,\Delta x}_\rho(t_n,x_m) +\widetilde L_\rho {\widetilde C} (1+(I^{h,\Delta x}_\rho(t_n,x_m))^{2M})\Big(h^{(M-1)/2M}+  \Delta x/h \Big){+\delta(x,\rho)}.
\end{split}
\end{equation}
\end{teo}

\begin{proof}
{Let us consider for simplicity the semi discrete case}. The lower bounds follow by \eqref{eq:ineqRHO} applying Proposition \ref{prop:lower} to the value function $v_\rho$.
As $A$ is bounded, the definition of $X^{t,x,\alpha}_\cdot$ from (\ref{eq:SDE}) satisfies the assumptions on the coefficients in
Lemma \ref{lem:large_dev}. We therefore get immediately
\begin{eqnarray*}
0 \le v(t,x)-v_\rho(t,x) &=& \sup_{\alpha\in\cA}\,\E_{t,x}\big[U(X^\alpha_T)\big] - \sup_{\alpha\in\cA}\,\E_{t,x}\big[U_\rho(X^\alpha_T)\big] \\
&\le& \sup_{\alpha\in\cA}\,\E_{t,x}\big[U(X^\alpha_T) - U_\rho(X^\alpha_T)  \big] \\
&\le& U(c_0/\rho) \sup_{\alpha\in\cA}\ \mathbb{P}\left[X^{t,x,\alpha}_T \le c_0/\rho \right] + 
U'(\rho) \sup_{\alpha\in\cA}\ \mathbb{E}_{t,x}\left[(X^\alpha_T-\rho) \mathbbm{1}_ {\{X^\alpha_T\ge \rho\}} \right] \\
&=& o(\rho^{-p})
\end{eqnarray*}
for all $p$ and all $x$.
Thanks to the duality property \eqref{eq:duality_CL_rho}, one has 
\begin{align*}
v(t_n,x) & \leq \inf_{y>0}\left\{ \widetilde v_\rho(t_n,y) + xy\right\} {+\delta(x,\rho)}.
\end{align*}
Applying Proposition \ref{prop:upper}, one has 
\begin{align*}
 \widetilde v_\rho(t_n,y)\leq  \widetilde V_\rho (t_n,y) + xy + \widetilde L_\rho \widetilde C (1 + y^{2M})h^{(M-1)/2M},
 \end{align*}
 so that arguing as in the proof of Theorem \ref{eq:error} we get the upper bounds.
\end{proof}

\begin{remark}
The previous error estimates clearly depend on the parameter $\rho$ and the utility function $U$ via the Lipschitz constant $L_\rho$.
In the case of power utility, we have $L_\rho = x_\rho^{p-1}/p = c_0^{p-1}/p \, \rho^{1-p}$. As $\delta(x,\rho)$ goes to zero faster than any power of $1/\rho$,
we can choose $\rho = h^{-r}$ for arbitrarily small positive $r$ and therefore obtain an order of $L_\rho h^{(M-1)/2M} + \delta(x,\rho)$ arbitrarily close to the Lipschitz case, i.e.\
$(M-1)/2M$.
\end{remark}

\begin{remark}
The above result can also be extended to cases where $\lim_{x\downarrow 0} U(x) = - \infty$, by considering
$U_\rho(x) = U(x_\rho) + U'(x_\rho) (x-x_\rho)$ for $x \in [0,x_\rho]$.
Then we can estimate $\mathbb{E}\left[(U_\rho(X_t)-U(X_t)) \mathbbm{1}_ {\{X_t \ge \rho\}} \right] $ similar to the proof of Lemma \ref{lem:large_dev},
as long as $U$ does not grow more than, e.g., exponentially in $-1/x$ as $x\rightarrow 0$. This is in particular satisfied by the commonly used log-utility.
\end{remark}
}

\section{Numerical tests}\label{sec:num}

We test our theoretical results on some concrete examples numerically. We consider $d=1$ and the computational domain  $[0,x_{\max}]$.
We denote by $N$ and $J$ the number of time and space steps, respectively, i.e.
$$
h = \frac{T}{N}\qquad \text{and}\qquad \Delta x = \frac{x_{\max}}{J}.
$$
We study the case of a power utility function:
\begin{eqnarray}
\label{power}
U(x)=\frac{x^p}{p}\qquad \text{for some } p\in (0,1).
\end{eqnarray}
We consider the modification $U_\rho$ of the utility function obtained in (\ref{LipU}), for $\rho = 18$ and $c_0 = 8$. The utility function $U$ for $p=0.5$  and its conjugate $\widetilde U$, as well as its Lipschitz continuous approximation $U_\rho$ and its conjugate $\widetilde U_\rho$ are shown in Figure \ref{fig:utilities}. 
\begin{figure}
\includegraphics[width=0.65\textwidth]{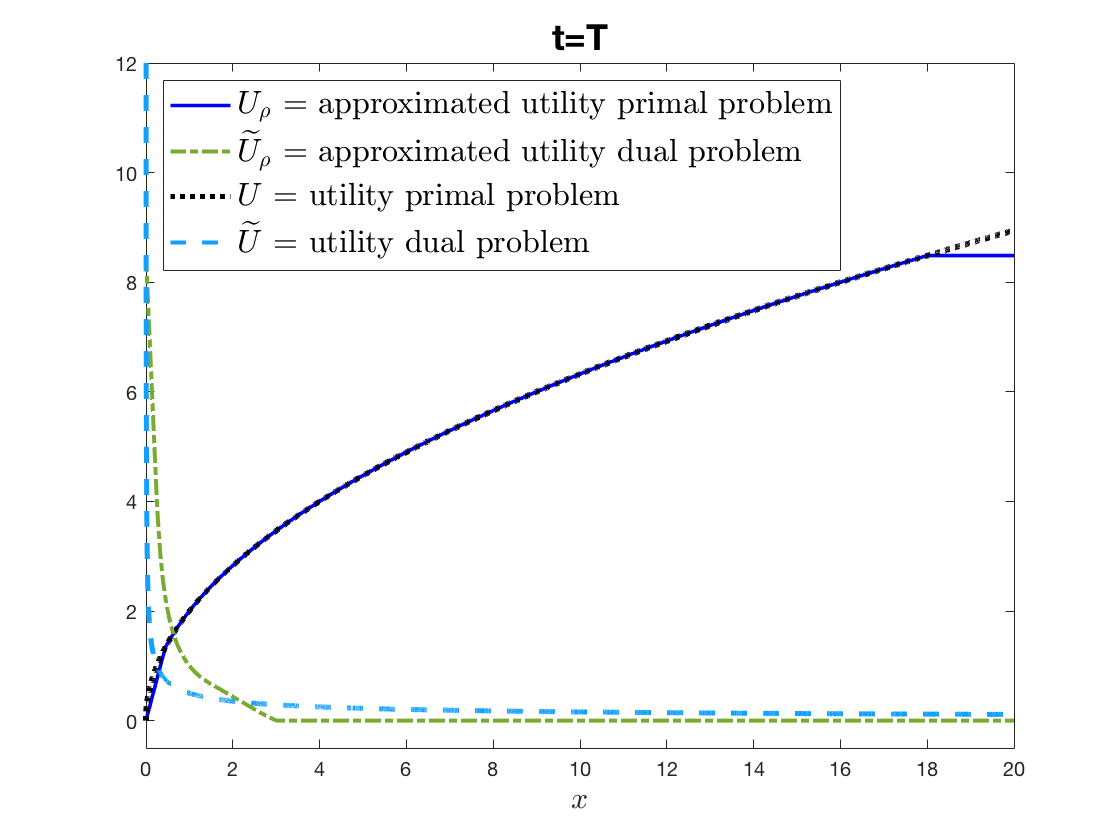}
\caption{
The power utility function $U$ (in dotted black) with its conjugate $\widetilde U$ (in dashed cyan) together with the Lipschitz continuous approximation $U_\rho$ (in solid blue) and its conjugate $\widetilde U_\rho$ (in dash-dotted green). Here, $x_{\max} =20$, $\rho = 18$ and $c_0 = 8$.}\label{fig:utilities}
\end{figure}


In our tests, we take $M=4$ with $\Delta x \sim h^{11/8}$ obtained from \eqref{eq:error_fullyRHO} { balancing the error terms}, more specifically $J\sim \lceil N^{11/8} \rceil$.
Taking $M>2$ has only (theoretical) advantages for non-smooth solutions,
while we would observe order of convergence at most one for any choice of $M\geq 2$, even in the smooth case.
This is due to the fact that, even in the case of smooth solutions, the use of the Euler-Maruyama scheme reduces the order of consistency of the overall scheme to one
(noting that a modified proof utilising the higher weak order 1 of the Euler-Maruyama scheme, compared to the strong order 1/2, can be used in the smooth case), 
regardless of the value of $M$. 
An improvement of the order of consistency might be achieved by combining higher values of $M$ with the use of higher order time-stepping schemes, for instance the higher order Taylor schemes of \cite{KloPla}.

For the optimization over the controls in our computations, we truncate $A$ and $\Gamma$ first to a finite interval{, if necessary,} and then discretise the interval
by $N_a$ and $N_\gamma$ equally spaced mesh points, respectively. As already pointed out in {Section \ref{sec:approx_dual} and Remark \ref{rem:controls}}, this further approximation decreases the value of the discrete primal (maximisation) problem and increases the value of the discrete dual (minimisation) problem, in the same way as the piecewise constant (in time) control approximation does. This implies that this component of the error is captured in the numerical duality gap which we compute \textit{a posteriori}.
The approximation can generally only be improved by increasing the size of the control intervals and decreasing the control mesh spacing, concurrently with decreasing $h$ and $\Delta x$.

As the optimal control in our examples is bounded, the error of the control truncation is zero if the interval is chosen large enough.
It is seen from the computations that the contribution of the control discretisation error is small, decreasing quadratically in $N_a^{-1}$ and $N_\gamma^{-1}$ since we have a smooth dependence of the Hamiltonian on the control.
In our tests, we take $N_a \sim N_\gamma \sim N$, such that the control discretisation error becomes eventually negligible.

{
As the point $x_m$ approaches $0$ or $x_{\max}$, it may happen that $\widehat X^{t_n,x_m}_{t_{n+1}}$ oversteps the domain $(0,x_{\max})$.
In this case,
we use linear extrapolation in order to define $W_\rho$ and $\widetilde W_\rho$ outside the computational mesh.
More precisely, one can verify that, due to the boundedness of the control and coefficients, the process $X_\cdot$ from (\ref{eq:SDE}) never reaches 0 for $x>0$ and equation (\ref{HJB}) holds up to the left boundary.
From equation (\ref{eq:SL_V_fully}), it is clear that for $m = 0$, the argument of the expression on the right-hand side is  $(t_{n+1},0)$,
so that $W_\rho(t_{n},0) =W_\rho (t_{n+1},0)$ for all $n$, at the boundary point.
For $m>0$ and $h$ small enough, the argument is $(t_{n+1},x)$ for some $x>0$. If $x<x_{\max}$, i.e.\ $x$ in some interval $(x_k,x_{k+1}]$, $k\ge 0$, in the interior of the domain,
the value can be obtained by linear interpolation from $W_\rho(t_{n+1},x_k)$ and $W_\rho(t_{n+1},x_{k+1})$.
In the rare case that $x<0$ (for larger $h$) we extend $W_\rho(t_{n+1},\cdot)$ in (\ref{eq:SL_V_fully}) by linear extrapolation from $[x_0,x_1]$ to negative $x$. As this is only needed for
$h$ above a certain threshold, it does not affect our estimates.
In the case $x>x_{\max}>\rho$ (where we choose $x_{\max}$ and $\rho$ so that the second inequality holds),
we can set $W_{\rho}(t_{n+1},x) = U_{\rho}(\rho)$, where we have exactly $v_{\rho}(t_{n+1},x)= U_{\rho}(\rho)$ if $x_{\max}$ is large enough because of the constancy of the solution for
large $x$. A similar argument holds for the dual problem.
}


\subsection*{Test 1: Merton problem}
We first study the classical Merton problem. This corresponds to the dynamics \eqref{eq:SDE} with $g\equiv 0$, constant coefficients $b,r,\sig$ and $A=\R$. It is well known that for this problem there exists a closed-form solution given by (see, e.g.\ \cite{Pham_book})
$$
v(t,x)=\exp\Big\{t\Big(a^*(b-r)+r-\frac{1}{2} (a^*)^2 (1-p)\sig^2\Big)\Big\} U(x),
$$
where $U$ and $p$ are given in (\ref{power}), and
$$
a^*:=\frac{(b-r)}{\sig^2(1-p)}
$$
is the optimal control. We recall that in this case the dual problem is linear and no optimisation is necessary since $\Gamma = \{ 0 \}$.
The values of the coefficients used in the test is given in  Table \ref{tab:data}. For these values, setting $A=[-1, 1]$ is sufficient to have $a^*\in A$.\\
\begin{table}[!hbtp]
\centering
\begin{tabular}{|c|c|c|c|c|c|}
\hline 
$p$ & $r$  & $b$ & $\sig$ &  $T$  & $x_{\max}$\\
\hline
$0.5$ & $0.8$ & $1.2$  & $1$ & $0.5$  & $20$\\
\hline
\end{tabular}
\caption{Test 1: Parameters used in numerical experiments.}
\label{tab:data}
\end{table}\\
Table \ref{tab:test1solExLocalN54} reports the error and the estimated convergence rate of $W_\rho$ to the exact solution $v$  of the primal problem. As expected, the order of convergence is around 1. It is important to notice that continuing to refine the mesh without increasing $\rho$, we cannot get convergence to $v$. In fact, the probability in  \eqref{eq:est_ld1} and \eqref{eq:est_ld2}, even if small at points $x$ far from the boundaries of the domain, is different from zero everywhere (see also Figure~\ref{fig:test1error}, left).
To reduce the contribution to the error coming from the term in $\rho$ we compute the error locally, away from the boundary of the computational domain.

In Table \ref{tab:test2Local}, we report the numerical duality gap, i.e.\ the quantity $G_\rho^{h,\Delta x}(T,x)$. This quantity also decreases with order 1 or even slightly higher. In this case, the duality gap is bigger than the error, but of the same order. 
In Figure \ref{fig:test1error} (right) we show the numerical solutions $W_\rho$ and $\widetilde W_\rho$ of the primal and the dual problem, together with the convex conjugate of $\widetilde W_\rho$. 

\begin{table}[!htbp]
\begin{tabular}{cc|ccccccc}
$J$ & $N$ & Error $L^1$ & Order $L^1$ & Error $L^2$ & Order $L^2$ & Error $L^\infty$ & Order $L^\infty$ & CPU (s)\\
\hline
  18 &    8  & 1.96E-01 &  -  & 1.86E-01 &  -  & 1.77E-01 &  - & 0.30\\
   46 &   16  & 1.44E-01 &  0.44  & 1.12E-01 &  0.74  & 1.05E-01 &  0.75 & 1.05 \\
  118 &   32  & 5.85E-02 &  1.30  & 4.54E-02 &  1.30  & 5.86E-02 &  0.84 & 3.91\\
  305 &   64  & 1.52E-02 &  1.94  & 1.14E-02 &  2.00  & 1.52E-02 &  1.95 & 15.54 \\
  790 &  128  & 5.70E-03 &  1.42  & 4.11E-03 &  1.47  & 4.76E-03 &  1.67 & 61.95 \\
 2048 &  256  & 2.35E-03 &  1.28  & 1.68E-03 &  1.29  & 1.74E-03 &  1.45 & 467.54\\
 5312 &  512  & 1.12E-03 &  1.07  & 8.14E-04 &  1.04  & 9.18E-04 &  0.92 &  2169.45\\

 \end{tabular}
 \caption{\label{tab:test1solExLocalN54} Test 1: Local ($x\in [1,2]$) errors and convergence order comparing $W_\rho$ with the exact solution $v$, for $M=4$ (Gau\ss-Hermite quadrature points), $N=4\cdot 2^{k}$ (time steps), $J = \lceil N^{11/8}\rceil $ (space steps), $N_a=2^{k} +1$ (discrete controls), for $k=1,2,\ldots, 8$.}
 \end{table}

\begin{table}[h]
\begin{tabular}{cc|cccccccc}
$J$ & $N$ &Gap $L^1$ & Order $L^1$ & Gap $L^2$ & Order $L^2$ & Gap $L^\infty$ & Order $L^\infty$ & CPU (s)\\
\hline 
 18 &    8  & 2.17E+01 &  -  & 7.17E+00 &  -  & 3.22E+00 &  - & 0.56\\
 46 &   16  & 1.24E+01 &  0.80  & 4.04E+00 &  0.83  & 1.65E+00 &  0.96 & 1.41\\
 118 &   32  & 7.24E+00 &  0.78  & 2.31E+00 &  0.80  & 9.24E-01 &  0.88 & 4.70\\
305 &   64  & 3.92E+00 &  0.89  & 1.26E+00 &  0.88  & 5.06E-01 &  0.87 & 17.98\\
 790 &  128  & 1.87E+00 &  1.07  & 6.03E-01 &  1.06  & 2.43E-01 &  1.06 & 110.56\\
2048 &  256  & 7.16E-01 &  1.38  & 2.37E-01 &  1.35  & 1.00E-01 &  1.28 & 656.69\\
 5312 &  512  & 1.72E-01 &  2.05  & 5.53E-02 &  2.10  & 2.20E-02 &  2.19 & 2813.47\\
 13778 & 1024  & 5.97E-02 &  1.53  & 1.94E-02 &  1.51  & 8.05E-03 &  1.45 & 17059.00
  \end{tabular}
\caption{\label{tab:test2Local} Test 1: Global ($x\in [0,x_{\max}]$) duality gap $G_\rho^{h,\Delta x}$ from (\ref{eqn:gap_fully})
and related convergence order, for $M=4$ (Gau\ss-Hermite quadrature points), $N=4\cdot 2^{k}$ (time steps), $J = \lceil N^{11/8}\rceil $ (space steps), $N_a= 2^{k}+1$ (discrete controls), for $k=1,2,\ldots, 8$.}
\end{table}
\begin{figure}
\includegraphics[width=0.47\textwidth]{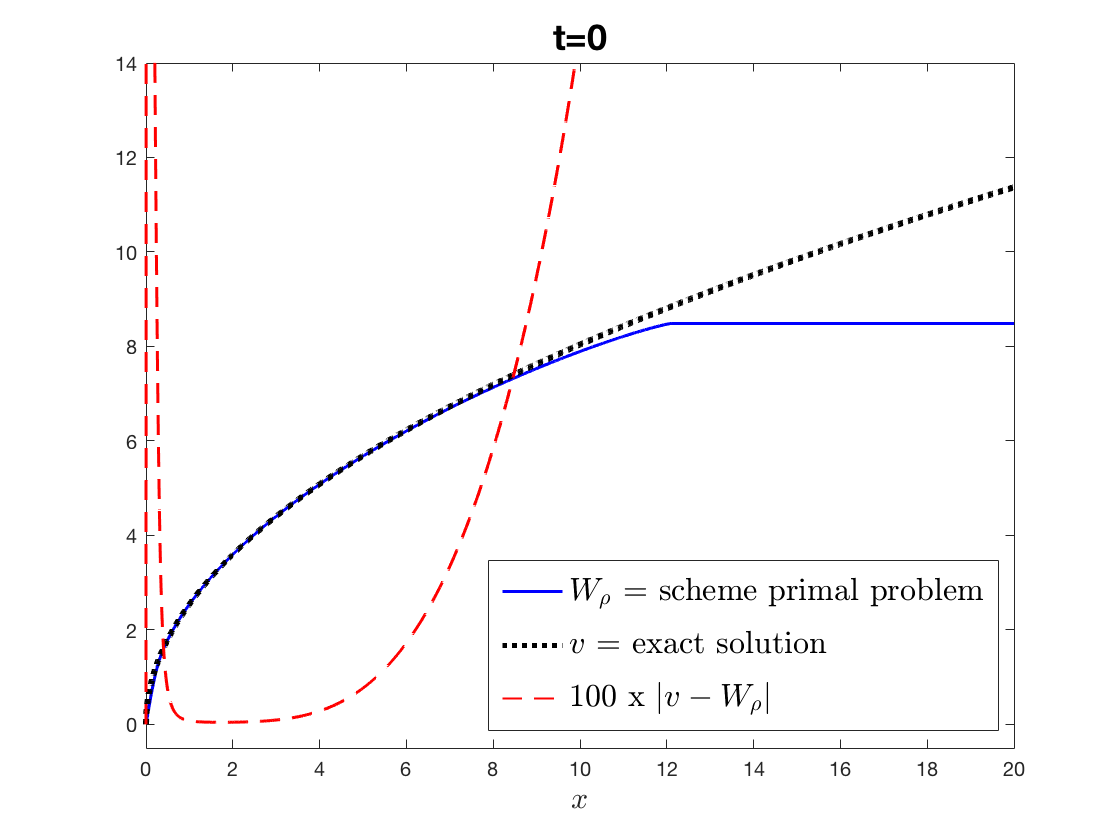}
\includegraphics[width=0.47\textwidth]{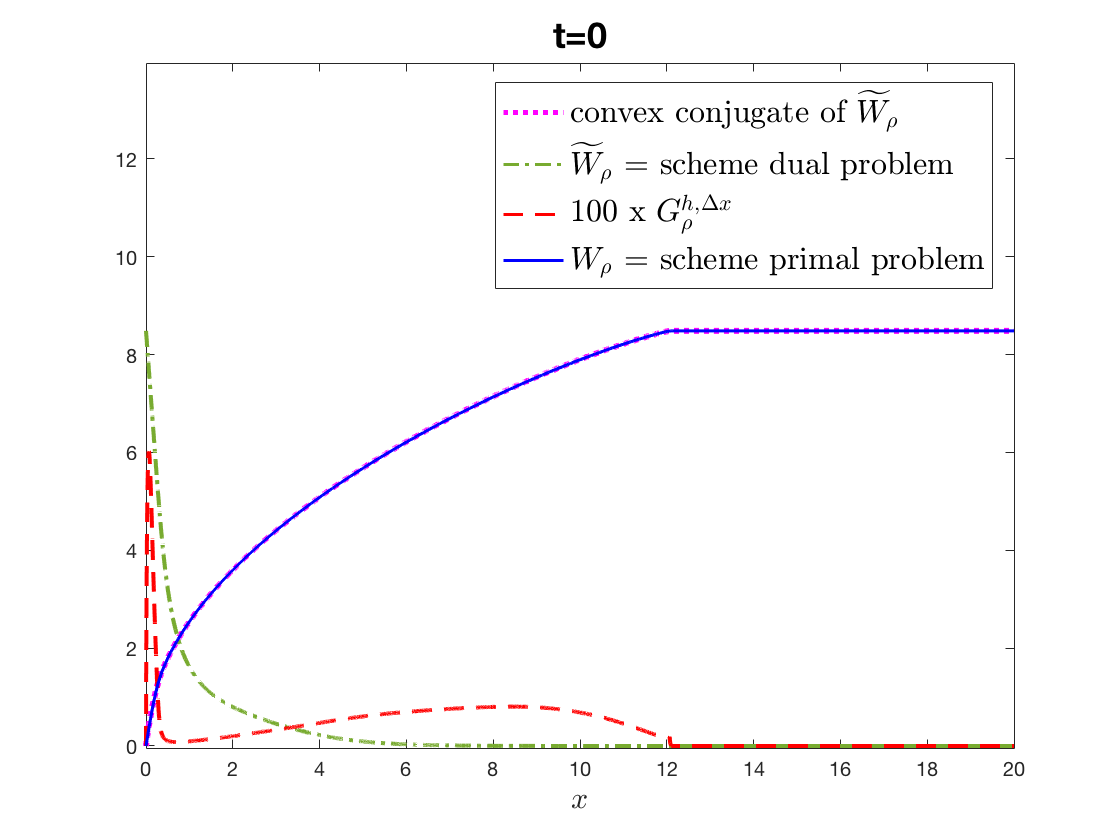}
\caption{Test 1: Numerical solution $W_\rho$ (in solid blue)  compared with the exact solution (dotted black, left) and the convex conjugate of $\widetilde W_\rho$ (dotted magenta, right). The dashed red line represents the error (left) and the numerical duality gap (right), multiplied by a factor 100. The dash-dotted green line on the right is the numerical approximation of the dual problem $\widetilde W_\rho$. }\label{fig:test1error}
\end{figure}

From the results in Table \ref{tab:test2Local} we deduce that (given the choice of $\Delta x$ in relation to $h$)
$$
 |G_\rho^{h,\Delta x}(t,x)|\leq C \left( h + \Delta x^{8/11}\right),
$$
which, combined with \eqref{eq:error_fully} and taking $M=4$, gives the \textit{a posteriori} bounds 
\begin{align}\label{eq:est_Merton}
-C \Big(h^{3/8}+  \Delta x^{3/11} \Big) \leq v(t_n,x_m)-W_\rho(t_n,x_m)\leq C \left(h + \Delta x^{8/11} + h^{3/8}+  \Delta x^{3/11}\right),
\end{align}
which in conclusion is a symmetric bound of order $3/8$  in time and $3/11$ in space.

{
For using our error estimates, it is necessary to solve numerically both the primal and the dual problem. The computational cost for the solution of the dual problem is comparable to that of the primal one, which has the same structure and uses the same scheme. This can be partially observed comparing the CPU times in Table \ref{tab:test1solExLocalN54}  and \ref{tab:test2Local} (however, in this case the dual problem is linear and the computational cost is less than double that of the primal one). }

We illustrate the different contributions to the error, together with the actual error, in Figure \ref{fig:errors}.
The figure shows the order (at least) one for the empirical error and for the numerical duality gap,
as one would have expected from the first order error of the scheme for sufficiently smooth solutions.
We also plot the theoretical error bounds, which hold in the general non-smooth case, for the Euler-Maruyama scheme, given by the expression \eqref{eq:EMest} in the Appendix, of  order $1/2$, and for the Gau\ss-Hermite approximation, from \eqref{eq:GHest}, of order $3/8$. The big constants appearing in the theoretical \textit{a priori} bounds, which are not sharp, put the magnitude of these theoretical errors far from that of the empirical one.

\begin{figure}
\includegraphics[width=0.7\textwidth]{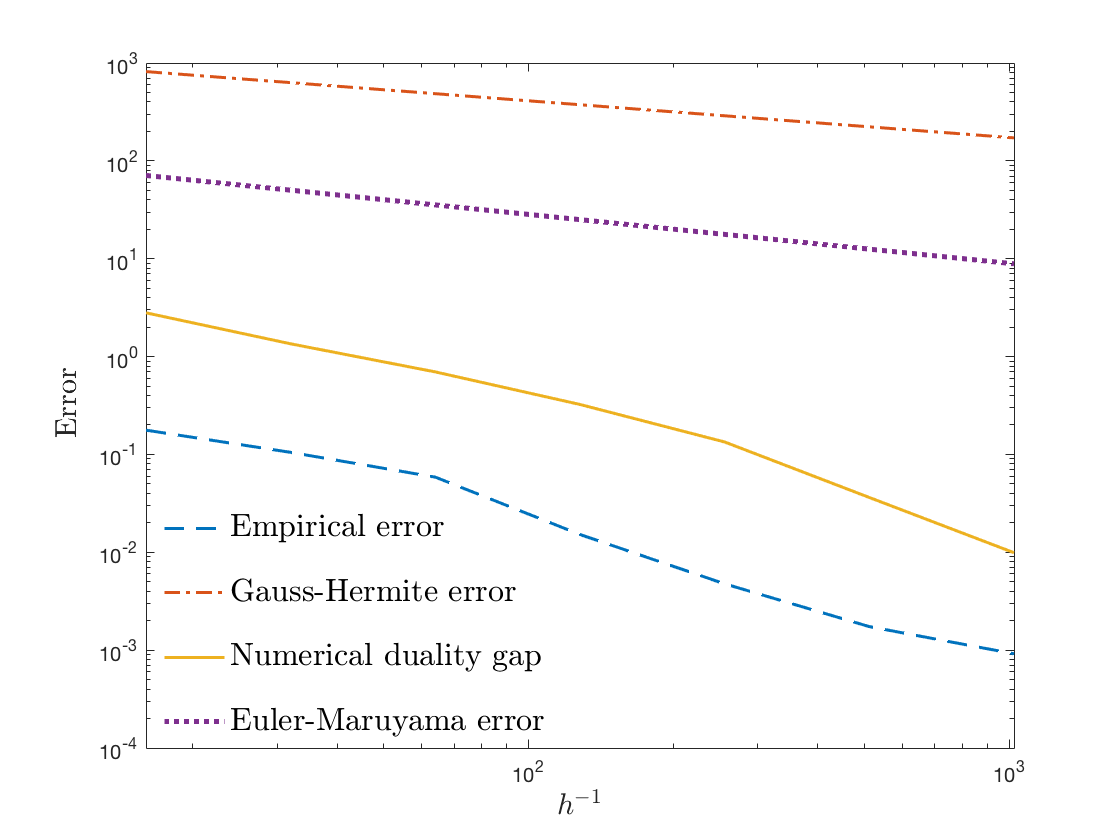}
\caption{Test 1. Local ($x\in [1,2]$) empirical error $|v(0,x)-W_\rho(0,x)|$ as reported in Table \ref{tab:test2Local}, global numerical duality gap $G_\rho^{h,\Delta x}$ reported in Table \ref{tab:test1solExLocalN54}, theoretical error estimate for the Euler-Maruyama and Gau\ss-Hermite approximation given by \eqref{eq:EMest} and \eqref{eq:GHest}, respectively.  }\label{fig:errors}
\end{figure}

For this problem, the optimal control is constant over time, so there is no error coming from the piecewise control approximation and theoretical bounds as those provided by \eqref{eq:EMest} and \eqref{eq:GHest} can be used for both the upper and lower bound. The numerical duality gap in this case contains the sum of the numerical approximation errors for the primal and the dual problem as well as the error coming from the approximation in $\rho$ and the computation of the numerical convex conjugate.

\subsection*{Test 2: Cuoco and Liu example}
This example is taken from \cite{CuoLiu00}. In this paper, the authors consider the nonlinear dynamics in \eqref{eq:SDE} (i.e.\ $g\not\equiv 0$) and portfolio constraints (i.e.\ {$A\subsetneq \R$}). We still consider a power utility and $d=1$. Let $A$ be defined by 
$$
A=\Big\{a\in \R :  \max(0,-a)\lambda_-+\max(0,a)\lambda_+ \leq 1\Big\}
$$
for some $\lambda_-\geq 0$ and $\lambda_+\in [0,1]$.
The function $g$ is defined by 
$$
g(a) = -r (1+ \iota \lambda_{-})\max(0,-a) - (R-r)\big(1-\max(0,a)-\iota \lambda_{-}\max(0,-a)\big),
$$
where $R\geq r$ and $\iota\in [0,1]$. The values used in our numerical simulation are reported in Table \ref{tab:datatest2}. 
\begin{table}
\centering
\begin{tabular}{|c|c|c|c|c|c|c|c|c|c|}
\hline 
$p$ & $r$  & $R$ & $b$ & $\sig$ &  $T$  & $x_{\max}$ & $\iota$ & $\lambda_+$ & $\lambda_-$\\
\hline
$0.5$ & $0.8$ & $1$ & $1.2$  & $0.5$ & $0.5$  & $20$ & $0.5$ & $1$ & $1$ \\
\hline
\end{tabular}
\caption{Test 2: Parameters used in numerical experiments.}
\label{tab:datatest2}
\end{table}

Observe that the choice $\lambda_+=\lambda_-=1$ corresponds to $A=[-1,1]$. In order to define $\Gamma$, we use the explicit expression given in \cite[Section 5.2]{CuoLiu00} for the optimal control. For the data in Table \ref{tab:datatest2}, we can take $\Gamma = [-1,1]$ to guarantee $\nu^*_t\in \Gamma$ for any $t\in [0,T]$.
Table \ref{tab:test2GlobalN54} 
reports the numerical duality gap and the corresponding convergence order. The numerical solutions $W_\rho$ and $\widetilde W_\rho$ of the primal and the dual problem, together with the convex conjugate of $\widetilde W_\rho$ are shown in Figure \ref{fig:test2gap}. 

\begin{table}
\begin{tabular}{cc|ccccccc}
$J$ & $N$ & Gap $L^1$ & Order $L^1$ & Gap $L^2$ & Order $L^2$ & Gap $L^\infty$ & Order $L^\infty$ & CPU (s)\\
\hline   
   18 &    8  & 2.26E+01 & -  & 7.44E+00 & -  & 3.59E+00 &  - & 0.79\\
46 &   16  & 1.09E+01 &  1.05  & 3.48E+00 &  1.10  & 1.47E+00 &  1.29 & 2.51 \\
118 &   32  & 5.59E+00 &  0.96  & 1.74E+00 &  1.00  & 6.87E-01 &  1.10 & 9.83\\
 305 &   64  & 2.82E+00 &  0.99  & 8.79E-01 &  0.99  & 3.47E-01 &  0.98 & 45.94\\
 790 &  128  & 1.38E+00 &  1.03  & 4.35E-01 &  1.01  & 1.77E-01 &  0.97 & 552.49\\
 2048 &  256  & 5.75E-01 &  1.26  & 1.83E-01 &  1.25  & 7.49E-02 &  1.24 & 6305.33\\
 5312 &  512  & 1.56E-01 &  1.88  & 5.00E-02 &  1.87  & 2.08E-02 &  1.85 & 54006.70
  \end{tabular}
\caption{\label{tab:test2GlobalN54}Test 2: Global ($x\in [0,x_{\max}]$) duality gap $G_\rho^{h,\Delta x}$ from (\ref{eqn:gap_fully}
and related convergence order, for $M=4$ (Gau\ss-Hermite quadrature points), $N=4\cdot 2^{k}$ (time steps), $J = \lceil N^{11/8}\rceil $ (space steps), $N_a= 2^{k}+1$ (discrete controls), for $k=1,2,\ldots, 8$.}
\end{table}


\begin{figure}
\includegraphics[width=0.6\textwidth]{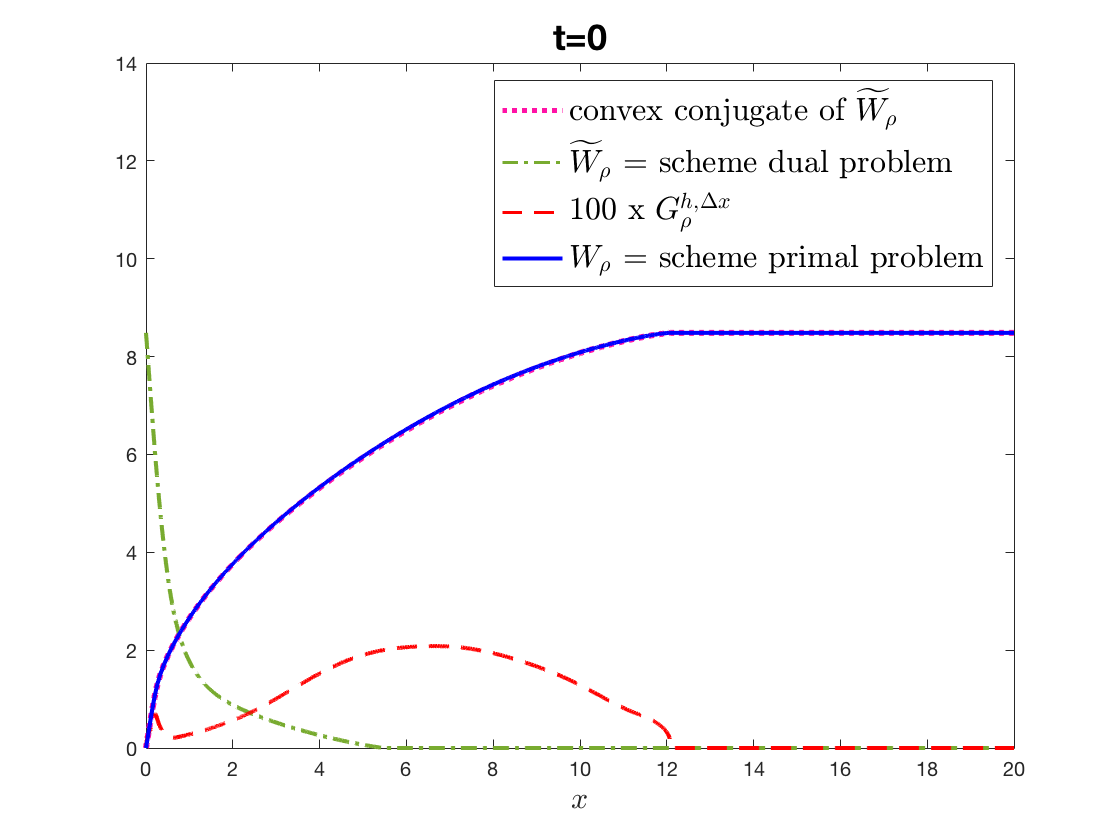}
\caption{Test 2: Numerical solution $W_\rho$ (in solid blue)  compared with the convex conjugate of $\widetilde W_\rho$ (dotted magenta). The  dashed red line represents the numerical duality gap multiplied by a factor 100 and the dash-dotted green line the numerical approximation of the dual problem $\widetilde W_\rho$.}\label{fig:test2gap}
\end{figure}

The results in Table \ref{tab:test2GlobalN54} give once again an estimate of the form 
$$
 |G_\rho^{h,\Delta x}(t,x)|\leq C \left( h + \Delta x^{8/11}\right)
$$
for the duality gap, leading to the \textit{a posteriori} bounds \eqref{eq:est_Merton}.

\section{Conclusion and perspectives} \label{concl}

For a suitable class of convex optimal control problems, we obtained in this paper  \textit{a posteriori} error  bounds using the numerical approximation of a dual problem.

Our numerical tests confirm the results given by the theoretical analysis and suggest a convergence to zero with order one of the numerical duality gap. Establishing rigorously a duality relation between the numerical approximations of the primal and the dual problem seems to us an interesting direction of research that we would like to pursue. Beyond the independent theoretical interest, this would also allow us to obtain an \textit{a priori} upper bound for the numerical error. The possibility of improving the order by higher order time stepping is also left for future research.

\appendix

\section{Explicit computation of the constants}\label{app}
In this section, we explicitly compute the constant $C$ which appears in the lower bound of \eqref{eq:bounds}. Analogous estimates can be used to derive the constant $\widetilde C$ appearing in the upper bound.
In what follows we denote for $t\in [0,T]$, $a\in A$, $x\in \R$:
\begin{align*}
\mu(t,x,a) := \left(r(t) + a^\top \cdot(b(t)-r(t)\mathbbm 1) + g(t,a)\right) x, &\qquad      \psi(t,x,a):=a^\top\sigma(t) x.
\end{align*}
Let $C_{\mu}, C_{\psi}\geq 0$ such that for $t,s \in [0,T]$, $a\in A$, $x,y\in \R$:
\begin{align*}
|\mu(t,x,a) - \mu(s,y,a)| & \leq C_{\mu} \left(|x-y|+(1+|x|)|t-s|^{1/2}\right), \\
  |\psi(t,x,a) - \psi(s,y,a)| & \leq C_{\psi} \left(|x-y|+(1+|x|)|t-s|^{1/2}\right)
\end{align*}
and 
\begin{align*}
|\mu(t,x,a)| \leq C_{\mu}(1+|x|),\quad & \quad |\psi(t,x,a)| \leq C_{\psi}(1+|x|).
\end{align*}

\subsection{Explicit bounds for the  Euler-Maruyama approximation}

We consider the Euler-Maruyama approximation given by \eqref{eq:def_M} for $\alpha\equiv(a_0,\ldots, a_{N-1})\in \mathcal A_h$. This leads to the following expression for $\overline X^{t_n,x,\alpha}_{\cdot}$:
$$
\overline X^{t_n,x,\alpha}_{t_k} = x + \sum^{k-1}_{i=n} \int^{t_{i+1}}_{t_i} \mu (t_i,\overline X^{t_n,x,\alpha}_{t_i},a_i) \, \mathrm{d}s + \int^{t_{i+1}}_{t_i} \psi (t_i,\overline X^{t_n,x,\alpha}_{t_i},a_i) \, \mathrm{d}W_s.
$$
Moreover, by the very definition of $X^{t_n,x,\alpha}_{\cdot}$:
$$
 X^{t_n,x,\alpha}_{t_k} = x + \sum^{k-1}_{i=n} \int^{t_{i+1}}_{t_i} \mu (s,X^{t_n,x,\alpha}_{s},a_i) \, \mathrm{d}s + \int^{t_{i+1}}_{t_i} \psi (s,X^{t_n,x,\alpha}_{s},a_i) \, \mathrm{d}W_s.
$$
Therefore, using the Cauchy-Schwartz inequality and It{\^o} isometry together with classical estimates, one has 
\begin{align*}
\mathbb E\left[ |\overline X^{t_n,x,\alpha}_{t_k} - X^{t_n,x,\alpha}_{t_k}|^2\right] &  \leq 2 T \sum^{k-1}_{i=n}  \E\left[ \int^{t_{i+1}}_{t_i} \left|\mu (t_i,\overline X^{t_n,x,\alpha}_{t_i},a_i) - \mu (s, X^{t_n,x,\alpha}_{s},a_i)\right|^2 \mathrm{d}s\right] \\
& + 2 \sum^{k-1}_{i=n}\E\left[\int^{t_{i+1}}_{t_i} \left|\psi (t_i,\overline X^{t_n,x,\alpha}_{t_i},a_i) - \psi (s,X^{t_n,x,\alpha}_{s},a_i)\right|^2\mathrm{d}s\right]\\
& \leq 8 K_1 h \sum^{k-1}_{i=n} \left( \mathbb E\left[ \left| \overline X^{t_n,x,\alpha}_{t_i} - X^{t_n,x,\alpha}_{t_i}\right|^2\right]  + h + h\mathbb E\left[\underset{s\in [t_i,t_{i+1}]}\sup \left| X^{t_n,x,\alpha}_{s}\right|^2 \right]\right.\\
& \quad+ \left.  \mathbb E\left[ \underset{s\in [t_i,t_{i+1}]}\sup\left| X^{t_n,x,\alpha}_{s} - X^{t_n,x,\alpha}_{t_i}\right|^2\right]  \right),
\end{align*}
where we denoted  $K_1:=  (C_{\mu}^2 T +C_{\psi}^2)$. 
By classical estimates on the process $X^{t_n,x}_\cdot$ and denoting $K_2(\xi):= ( C^2_\mu \xi +4 C^2_\psi) $, one has 
\begin{align*}
\mathbb E\left[ \underset{s\in [t_i,t_{i+1}]}\sup \left| X^{t_n,x,\alpha}_{s}\right|^2 \right]& \leq 3\left(|x|^2 + 2K_2(T) \right) e^{6K_2(T) T}, \\
\mathbb E\left[\underset{s\in [t_i,t_{i+1}]}\sup\left| X^{t_n,x,\alpha}_{s} - X^{t_n,x,\alpha}_{t_i}\right|^2\right]& \leq  4 K_2(h) h \left(1+ 3\left(|x|^2 + 4K_2(T) \right) e^{ 6 K_2(T) T}\right).
\end{align*}
Putting these estimates together:
\begin{align*}
\mathbb E\left[ |\overline X^{t_n,x,\alpha}_{t_k} - X^{t_n,x,\alpha}_{t_k}|^2\right] & \leq 8 K_1 h \sum^{k-1}_{i=n} \mathbb E\left[ \left| \overline X^{t_n,x,\alpha}_{t_i} - X^{t_n,x,\alpha}_{t_i}\right|^2\right] 
 +  8 K_1 T h (1+ 2K_2(h)) (1+K_3 (x))
\end{align*}
with $K_3(x):= 3 \left(|x|^2 + 2K_2(T) T\right) e^{6K_2(T) T} $, so that, using Gronwall's lemma, one obtains
\begin{align*}
\mathbb E\left[ |\overline X^{t_n,x,\alpha}_{t_k} - X^{t_n,x,\alpha}_{t_k}|^2\right] & \leq  8 K_1 h (1+ 2K_2(h)) (1+K_3 (x))\left(1+e^{\sum^{k-1}_{i=n} 8 K_1 h}  \left(\sum^{k-1}_{i=n} 8 K_1 h \right)\right) \\
&\leq 8 K_1 h (1+ 2 K_2(h)) (1+K_3 (x))\left(1+ 8 K_1 T e^{ 8 K_1 T}  \right).
\end{align*}
Using the Lipschitz continuity of $U$, one has 
$$
\left|\underset{\alpha\in\mathcal A_h}\sup \mathbb E\left[ | U(\overline X^{t_n,x,\alpha}_{T})\right]  -\underset{\alpha\in\mathcal A_h}\sup\mathbb E \left[U(X^{t_n,x,\alpha}_{T})\right]\right| \leq L \underset{\alpha\in\mathcal A_h}\sup  \mathbb E\left[ |\overline X^{t_n,x,\alpha}_{T} - X^{t_n,x,\alpha}_{T}|\right]. 
$$
In conclusion, the contribution to the error coming from the Euler-Maruyama approximation can be bounded by 
$$
L  \Big(8 K_1(1+ 2K_2(h)) (1+K_3 (x))(1+ 8 K_1 T e^{ 8 K_1 T} ) \Big)^{1/2}\,h^{1/2}.
$$

For a linear (in the state), time independent dynamics as the one considered in Section \ref{sec:num}, one simply has 
\begin{align*}
|\mu(x,a) - \mu(y,a)|  \leq C_{\mu} |x-y|,\quad  & \quad|\psi(x,a) - \psi(y,a)|  \leq C_{\psi} |x-y|
\end{align*}
and 
\begin{align*}
|\mu(x,a)| \leq C_{\mu}|x|, \quad & \quad |\psi(x,a)| \leq C_{\psi}|x|.
\end{align*}
It is possible to verify that this leads to 
\begin{align*}
\mathbb E\left[ |\overline X^{t_n,x,\alpha}_{t_k} - X^{t_n,x,\alpha}_{t_k}|^2\right] 
& \leq 4 K_1 h \sum^{k-1}_{i=n} \left( \mathbb E\left[ \left| \overline X^{t_n,x,\alpha}_{t_i} - X^{t_n,x,\alpha}_{t_i}\right|^2\right]  +   \mathbb E\left[ \underset{s\in [t_i,t_{i+1}]}\sup\left| X^{t_n,x,\alpha}_{s} - X^{t_n,x,\alpha}_{t_i}\right|^2\right]  \right)\\
& \leq 4 K_1 h \sum^{k-1}_{i=n} \left( \mathbb E\left[ \left| \overline X^{t_n,x,\alpha}_{t_i} - X^{t_n,x,\alpha}_{t_i}\right|^2\right]  +   2 K_2(h)h\mathbb E\left[ \underset{s\in [t_i,t_{i+1}]}\sup\left| X^{t_n,x,\alpha}_{s} \right|^2\right]  \right)
\end{align*}
with 
$$
\mathbb E\left[ \underset{s\in [t_i,t_{i+1}]}\sup\left| X^{t_n,x,\alpha}_{s} \right|^2\right] \leq 3 |x|^2 e^{3 K_2(T)T} .
$$
Neglecting the infinitesimal terms,  one has
\begin{align*}
\mathbb E\left[ |\overline X^{t_n,x,\alpha}_{t_k} - X^{t_n,x,\alpha}_{t_k}|^2\right] 
& \leq 4 K_1 h \left(\sum^{k-1}_{i=n}  \mathbb E\left[ \left| \overline X^{t_n,x,\alpha}_{t_i} - X^{t_n,x,\alpha}_{t_i}\right|^2\right]  +   24T C_\psi^2 |x|^2e^{3 K_2(T)T}\right)\end{align*}
which leads to the sharper error estimate 
\beno
L   \left( 96 K_1 T C_\psi^2 |x|^2e^{3 K_2(T)T} \left( 1+ 4K_1 T e^{4K_1 T}\right)\right)^{1/2} \,h^{1/2}.
\eeno
In the estimates plotted in Section \ref{sec:num}, we consider
\be\label{eq:EMest}
L  \left( 24 K_1 T C_\psi^2 |x|^2 \left( 1+ 4K_1 T e^{4K_1 T}\right)\right)^{1/2} \,h^{1/2}
\ee
since we can approximate the second order moment of $X_\cdot $ by $x^2$ for  a local error.

\subsection{Explicit bounds for the  Gau\ss-Hermite approximation }
We consider the case of a one-dimensional Brownian motion.
Given a function $f\in C^{2M}(\R)$, the analysis in \cite[Proposition 3.2]{PicaReisi18_prob} shows that 

\begin{align*}
& \Big|\;\E_{t_n,x}\Big[f(\overline X^{a}_{t_{n+1}})\Big]-\E_{t_n,x}\Big[f(\widehat X^{a}_{t_{n+1}})\Big]\;\Big|\\
& \leq \Big |\int^{+\infty}_{-\infty}\frac{f^{(2M)}(\hat z)}{(2M)!}(\sqrt{2h}\psi(t_n,x,a)y)^{2M} \frac{e^{-y^2}}{\sqrt{\pi}}dy-\sum^M_{i=1}\lambda_i\frac{f^{(2M)}(\tilde z)}{(2M)!}(\sqrt{h}\psi(t_n,x,a) \xi_i)^{2M}\Big|\\
& \leq 2\|f^{2M}\|_\infty \frac{(2h)^M}{2M!} (\psi(t_n,x,a))^{2M} \int^\infty_{-\infty} y^{2M} \frac{e^{-y^2}}{\sqrt{\pi}} \mathrm{d}y\\
& \quad+ \|f^{2M}\|_\infty \frac{h^M}{2M!} (\psi(t_n,x,a))^{2M}\left| 2^M \int^\infty_{-\infty} y^{2M} \frac{e^{-y^2}}{\sqrt{\pi}} \mathrm{d}y - \sum^{M}_{i=1}\lambda_i{\xi_i}^{2M} \right|\\
& \leq \|f^{2M}\|_\infty \frac{h^M}{2M!} C_{\psi}^{2M}(1+|x|)^{2M}\left( 2(2M-1)!! + \Big|  (2M-1)!! - \sum^{M}_{i=1}\lambda_i {\xi_i}^{2M} \Big|\right),
\end{align*}
where in the last inequality we have used  that 
$$
2^M \int^\infty_{-\infty} y^{2M} \frac{e^{-y^2}}{\sqrt{\pi}} \mathrm{d}y = (2M-1)!!
$$
The estimate above corresponds to the error associated with the Gau\ss-Hermite approximation  at each time step, i.e.\ considering the error at time $t_{n+1}$ starting from $t_n$.
Our scheme being iterative in time, the overall contribution to the error will be 
$$
\left\|\frac{\partial  f}{\partial x^{2M}}\right\|_\infty \frac{h^{M-1}}{2M!} 2^{2M-1} C_{\psi}^{2M} \left( (2M-1)!! + \Big|  (2M-1)!! - \sum^{M}_{i=1}\frac{\omega_i}{\sqrt\pi} {z_i}^{2M} \Big|\right) \Big(1 + \sup_{\substack{\alpha\in\mathcal A_h\\ k=n\dots N}}\mathbb E_{t_n,x}\left[(\widehat X^{\alpha}_{t_k})^{2M}\right]\Big),
$$
where we also used the classical inequality $|a+b|^{2M}\leq 2^{2M-1} (|a|^{2M} + |b|^{2M})$.
It remains to estimate $\mathbb E_{t_n,x}\left[ (\widehat X^{\alpha}_{t_k})^{2M}\right]$. 
By the recursive definition of $\widehat X$, one has for any $k=n,\ldots, N$ 
\begin{align*}
\E_{t_n,x}\left[ (\widehat X^{\alpha}_{t_{k+1}})^{2M}\right] & = \E_{t_n,x}\left[ \left(  \widehat X^{\alpha}_{t_k} + h \mu(t_k, \widehat X_k,a_k) +\sqrt{h} \psi(t_k,\widehat X_k,a_k) \zeta_k\right)^{2M} \right]\\
& =  \E_{t_n,x}\left[ \sum^{2M}_{i=0} \sum^{i}_{j=0} {2M\choose i} {i\choose j} h^{i-j} (\widehat X^{\alpha}_{t_k})^{j} (\mu(t_k, \widehat X_k,a_k) )^{i-j} \left(\sqrt{h} \psi(t_k,\widehat X_k,a_k) \zeta_k\right)^{2M-i}\right]\\
& =  \E_{t_n,x}\left[ \sum^{M}_{i=0} \sum^{2i}_{j=0} {2M\choose 2i} {2i\choose j} h^{2i-j} (\widehat X^{\alpha}_{t_k})^{j} (\mu(t_k, \widehat X_k,a_k) )^{2i-j} \left(\sqrt{h} \psi(t_k,\widehat X_k,a_k) \zeta_k\right)^{2M-2i}\right],
\end{align*}
where the last equality follows observing that $\E[(\ldots) \zeta_k^{2j+1}]=0$ for $j=0,\ldots,M-1$ for any quantity, represented by ``$(\ldots)$", independent of $\zeta_k$. Therefore, thanks to the linear growth of $\mu$ and $\psi$ (taking for simplicity $C_1:=\max(C_\mu,C_\psi)$):
\begin{align*}
\E_{t_n,x}\left[ (\widehat X^{\alpha}_{t_{k+1}})^{2M}\right] &  =  \E_{t_n,x}\left[ \sum^{M}_{i=0} \sum^{2i}_{j=0} {2M\choose 2i} {2i\choose j} h^{M+i-j} C_1^{2M-j}  (\widehat X^{\alpha}_{t_k})^{j} (1+|\widehat X^{\alpha}_{t_k}|)^{2i-j} \zeta_k^{2M-2i}\right]\\
 &  \leq  \sum^{M}_{i=0} \sum^{2i}_{j=0} {2M\choose 2i} {2i\choose j} h^{M+i-j} C_1^{2M-j}  \E_{t_n,x}\left[ |\widehat X^{\alpha}_{t_k}|^{j} (1+|\widehat X^{\alpha}_{t_k}|)^{2i-j}  \right]\E\left[\zeta_k^{2M-2i}\right]\\
 & = \E_{t_n,x}\left[(\widehat X^{\alpha}_{t_k})^{2M}\right] +  \sum^{2M-1}_{j=0} {2M\choose j} h^{2M-j} C_1^{2M-j}  \E_{t_n,x}\left[ |\widehat X^{\alpha}_{t_k}|^{j} (1+|\widehat X^{\alpha}_{t_k}|)^{2M-j}  \right]\\
 & \quad +  \sum^{M-1}_{i=0} \sum^{2i}_{j=0} {2M\choose 2i} {2i\choose j} h^{M+i-j} C_1^{2M-j}  \E_{t_n,x}\left[ |\widehat X^{\alpha}_{t_k}|^{j} (1+|\widehat X^{\alpha}_{t_k}|)^{2i-j}  \right]\E\left[\zeta_k^{2M-2i}\right].
\end{align*}
For $0\leq i\leq M$ and $0\leq j\leq 2i$, one has 
$$
\E_{t_n,x}\left[ |\widehat X^{\alpha}_{t_k}|^{j} (1+|\widehat X^{\alpha}_{t_k}|)^{2i-j}  \right]\leq  2^{2i} \left( 1+ \E_{t_n,x}\left[ (\widehat X^{\alpha}_{t_k})^{2M}\right]\right).
$$
This gives:
\begin{align*}
\E_{t_n,x}\left[ (\widehat X^{\alpha}_{t_{k+1}})^{2M}\right] &  \leq  \E_{t_n,x}\left[(\widehat X^{\alpha}_{t_k})^{2M}\right]  + 
\left( 1+ \E_{t_n,x}\left[ (\widehat X^{\alpha}_{t_k})^{2M}\right]\right) h \left\{ \sum^{2M-1}_{j=0} {2M\choose j} h^{2M-j-1} C_1^{2M-j} 2^{2M} 
 \right. \\
 & \quad \left.+  \sum^{M-1}_{i=0} \sum^{2i}_{j=0} {2M\choose 2i} {2i\choose j} h^{M+i-j-1} C_1^{2M-j} 2^{2i}\E\left[\zeta_k^{2M-2i}\right]\right\}.
\end{align*}
Neglecting the infinitesimal terms and denoting
$$
K_4:=  2M C_1 2^{2M} +  \frac{2M(2M-1)}{2}C_1^{2} 2^{2M-2},
$$
we have 
$$
\E_{t_n,x}\left[ (\widehat X^{\alpha}_{t_{k+1}})^{2M}\right] \leq (1+ K_4 h) \E_{t_n,x}\left[ (\widehat X^{\alpha}_{t_{k}})^{2M}\right] + K_4 h.
$$
Iterating, this leads to 
$$
\E_{t_n,x}\left[ (\widehat X^{\alpha}_{t_{k+1}})^{2M}\right] \leq (1+K_4 h)^k x^{2M} + k h K_4
$$
for any $n\leq k \leq N-1$, with $K_4$ not depending on $k$ and $\alpha\in\mathcal A_h$. Therefore, we can conclude that
$$
\sup_{\substack{\alpha\in\mathcal A_h\\ k=n\ldots N}}\E_{t_n,x}\left[ (\widehat X^{\alpha}_{t_{k}})^{2M}\right] \leq x^{2M} e^{K_4 T} + K_4 T.
$$
To avoid an exponential growth in $M$ of the constants and motivated by the fact that in Section \ref{sec:num} we empirically computed a local error, we can strongly simplify our estimates by approximating
$$
\sup_{\substack{\alpha\in\mathcal A_h\\ k=n\ldots N}} \E_{t_n,x}\left[ (\widehat X^{\alpha}_{t_{k}})^{2M}\right]\approx x^{2M}.
$$
The presence of the $2M$-th derivative in the error bound requires to pass by  a mollification of the original value function. For a given regularization parameter $\varepsilon$ and mollified value function $v_\varepsilon$ it is possible to show that an estimate of the form 
$$
\left\|\frac{\partial^{2M} v_\varepsilon}{\partial x^{2M}}\right\|_\infty \leq L K_5 \varepsilon^{1-2M}
$$
 holds with $K_5:= ( 3 + 9K_1T e^{ 3K_1 T})^{1/2}$. The balancing between the  Gau\ss-Hermite and regularization error (the last one giving an extra error term of order $\varepsilon$) leads to the choice of optimal order $\varepsilon = h^{ (M-1)/2M}$.
Therefore, we get
\be\label{eq:GHest}
L K_5 h^{(M-1)/2M}\frac{2^{2M-1}}{2M!}  C_{\psi}^{2M} \left( (2M-1)!! + \Big|  (2M-1)!! - \sum^{M}_{i=1}\frac{\omega_i}{\sqrt\pi} {z_i}^{2M} \Big|\right) (1+x^{2M}).
\ee

\bibliography{biblio.bib}

\begin{thebibliography}{10}

\bibitem{BJ02}
G.~Barles and E.R. Jakobsen.
\newblock {O}n the convergence rate of approximation schemes for
  {H}amilton-{J}acobi-{B}ellman equations.
\newblock {\em M2AN Math. Model. Numer. Anal.}, 36:33--54, 2002.

\bibitem{BJ05}
G.~Barles and E.R. Jakobsen.
\newblock {E}rror bounds for monotone approximation schemes for
  {H}amilton-{J}acobi-{B}ellman equations.
\newblock {\em SIAM J. Numer. Anal.}, 43(2):540--558, 2005.

\bibitem{BJ07}
G.~Barles and E.R. Jakobsen.
\newblock {E}rror bounds for monotone approximation schemes for parabolic
  {H}amilton-{J}acobi-{B}ellman equations.
\newblock {\em Math. Comput.}, 74(260):1861--1893, 2007.

\bibitem{BS91}
G.~Barles and P.E. Souganidis.
\newblock {C}onvergence of approximation schemes for fully nonlinear second
  order equations.
\newblock {\em Asymptotic Anal.}, 4:271--283, 1991.

\bibitem{Beyer87}
W.H. Beyer.
\newblock {\em CRC Standard Mathematical Tables}.
\newblock CRC Press, 28th edition, 1987.

\bibitem{CF95}
F.~Camilli and M.~Falcone.
\newblock {A}n approximation scheme for the optimal control of diffusion
  processes.
\newblock {\em RAIRO Mod\'el. Math. Anal. Num\'er.}, 29(1):97--122, 1995.

\bibitem{CIL92}
M.G. Crandall, H.~Ishii, and P.L. Lions.
\newblock {U}ser's guide to viscosity solutions of second order partial
  differential equations.
\newblock {\em Bull. Amer. Math. Soc.}, 27(1):1--67, 1992.

\bibitem{CuoCvi98}
D.~Cuoco and J.~Cvitani\`c.
\newblock Optimal consumption choices for a `large' investor.
\newblock {\em J. Econ. Dyn. Control}, 22(3):401--436, 1998.

\bibitem{CuoLiu00}
D.~Cuoco and H.~Liu.
\newblock A martingale characterization of consumption choices and hedging
  costs with margin requirements.
\newblock {\em Math. Finance}, 10:355--385, 2000.

\bibitem{DJ12}
K.~Debrabant and E.R. Jakobsen.
\newblock {S}emi-{L}agrangian schemes for linear and fully non-linear diffusion
  equations.
\newblock {\em Math. Comp.}, 82(283):1433--1462, 2012.

\bibitem{debrabant2013semi}
K.~Debrabant and E.R. Jakobsen.
\newblock {S}emi-{L}agrangian schemes for parabolic equations.
\newblock In T.~Gerstner and P.~Kloeden, editors, {\em Recent Developments in
  Computational Finance: Foundations, Algorithms and Applications}, pages
  279--297. World Scientific, 2013.

\bibitem{ElKarPengQue97}
N.~El~Karoui, S.~Peng, and M.~C. Quenez.
\newblock Backward stochastic differential equations in finance.
\newblock {\em Math. Finance}, 7(1):1--71, 1997.

\bibitem{FerrFalcBook}
M.~Falcone and R.~Ferretti.
\newblock {\em Semi-Lagrangian Approximation Schemes for Linear and
  Hamilton-Jacobi Equations}, volume 133.
\newblock SIAM, Philadelphia, 2014.

\bibitem{Hildebrand56}
F.B. Hildebrand.
\newblock {\em Introduction to {N}umerical {A}nalysis}.
\newblock New York: McGraw-Hill, 1956.

\bibitem{hu2018existence}
Ying Hu and Shanjian Tang.
\newblock Existence of solution to scalar {BSDEs } with $l
  \exp\left(\sqrt{2\over\lambda \log (1+ L)} \right)$-integrable terminal
  values.
\newblock {\em Electron. Commun. Prob.}, 23, 2018.

\bibitem{JakoPicaReisi18}
E.R. Jakobsen, A.~Picarelli, and C.~Reisinger.
\newblock Improved order 1/4 convergence for piecewise constant policy
  approximation of stochastic control problems.
\newblock {\em Electron. Commun. Prob.}, 24(59):1--10, 2019.

\bibitem{KloPla}
P.E. Kloeden and E.~Platen.
\newblock {\em Numerical Solution of Stochastic Differential Equations}.
\newblock Berlin, New York, Springer-Verlag, 1992.

\bibitem{KraScha99}
D.~Kramkov and W.~Schachermayer.
\newblock {T}he asymptotic elasticity of utility functions and optimal
  investment in incomplete markets.
\newblock {\em Ann. Appl. Probab.}, 9(3):904--950, 1999.

\bibitem{K97}
N.V. Krylov.
\newblock {O}n the rate of convergence of finite-difference approximations for
  {B}ellman's equations.
\newblock {\em St. Petersburg Math. J.}, 9:639--650, 1997.

\bibitem{Kry99}
N.V. Krylov.
\newblock {A}pproximating value functions for controlled degenerate diffusion
  processes by using piece-wise constant policies.
\newblock {\em Electron. J. Probab.}, 4(2):1--19, 1999.

\bibitem{K00}
N.V. Krylov.
\newblock {O}n the rate of convergence of finite-difference approximations for
  {B}ellman's equations with variable coefficients.
\newblock {\em Probab. Theory Relat. Fields}, 117:1--16, 2000.

\bibitem{M89}
J.L. Menaldi.
\newblock {S}ome estimates for finite difference approximations.
\newblock {\em SIAM J. Control Optim.}, 27:579--607, 1989.

\bibitem{Merton}
R.C. Merton.
\newblock {O}ptimal consumption and portfolio rules in continuous time.
\newblock {\em J. Economic Theory}, 3:373--413, 1971.

\bibitem{Pham_book}
H.~Pham.
\newblock {\em Continuous-time Stochastic Control and Optimization with
  Financial Applications}, volume~61.
\newblock Series Stochastic Modeling and Applied Probability, Springer, 2009.

\bibitem{PicaReisi18_prob}
A.~Picarelli and C.~Reisinger.
\newblock Probabilistic error analysis for some approximation schemes to
  optimal control problems.
\newblock \emph{Systems \& Control Letters}. Forthcoming, available at
  arXiv:1810.04691, 2019.

\bibitem{Rogers}
L.C.G. Rogers.
\newblock {\em {D}uality in {C}onstrained {O}ptimal {I}nvestment {P}roblems:
  {A} {S}ynthesis}.
\newblock Number 1814 in Paris--Princeton Lectures on Mathematical Finance.
  Springer, 2002.

\bibitem{YZ}
J.~Yong and X.Y. Zhou.
\newblock {\em Stochastic Controls: Hamiltonian Systems and HJB Equations},
  volume~43 of {\em Applications of Mathematics}.
\newblock Springer-Verlag, New York, 1999.

\end{thebibliography}
\bibliographystyle{plain}


\end{document}